\documentclass[a4paper]{amsart}
\usepackage[margin=1.2in]{geometry}

\usepackage{colonequals}

\usepackage[T1]{fontenc} 
\usepackage{microtype}
\usepackage[american]{babel}
\usepackage{pdfpages}
\usepackage{emptypage}
\usepackage{geometry}
\usepackage{graphicx}

\usepackage{listings}
\usepackage{color}
\usepackage{xcolor}
\usepackage{listings}

\usepackage{caption}
\DeclareCaptionFont{white}{\color{white}}
\DeclareCaptionFormat{listing}{\colorbox{gray}{\parbox{\textwidth}{#1#2#3}}}
\usepackage[normalem]{ulem}

\usepackage{amsmath}
\usepackage{mathtools}
\usepackage{amsthm}
\usepackage{amssymb}
\usepackage{enumitem}

\usepackage{hyperref}
\usepackage[style=ieee-alphabetic,citestyle=alphabetic, backend=bibtex,isbn=false,url=false, maxbibnames=99]{biblatex}

\usepackage{tikz-cd}
\usepackage{chngcntr}

\usepackage{amsfonts}
\usepackage{yfonts}
\usepackage{mathrsfs}
\usepackage{bbm}
\usepackage[%
  cal=euler,
  scr=rsfso%
]{mathalfa} 

\usepackage{pdfcomment}

\usetikzlibrary{decorations.pathmorphing}

\DeclareMathOperator{\im}{im}

\DeclareMathOperator{\spec}{Spec}

\DeclareMathOperator{\Proj}{\underline{Proj}}
\DeclareMathOperator{\id}{id}

\DeclareMathOperator{\aut}{Aut}
\DeclareMathOperator{\autu}{\underline{Aut}}

\DeclareMathOperator{\pic}{Pic}

\DeclareMathOperator{\coker}{coker}

\DeclareMathOperator{\gl}{GL}

\DeclareMathOperator{\rk}{rk}

\DeclareMathOperator{\sym}{Sym}
\DeclareMathOperator{\spin}{Sp}
\DeclareMathOperator{\arf}{Arf}

\DeclareMathOperator{\chr}{char}
\DeclareMathOperator{\ann}{Ann}
\DeclareMathOperator{\isom}{Iso}

\DeclareMathOperator{\gr}{Gr}

\DeclareMathOperator{\syzy}{syzy}
\DeclareMathOperator{\Syzy}{Syzy}

\usepackage{array}   
\newcolumntype{M}{>{$}c<{$}} 
\newcolumntype{L}{>{$}l<{$}}
\usepackage{arydshln}
\usepackage{multirow}

\newcommand{\coarse}{\mathfrak{coarse}}

\DeclareMathOperator{\Hom}{\mathscr{H}\kern -2pt om}
\DeclareMathOperator{\Ext}{\mathscr{E}\kern -1.5pt xt}

\newtheorem{theorem}{Theorem}[section]
\numberwithin{theorem}{section} 

\newcommand{\nameofthm}{}

\newtheorem{genericThm}[theorem]{\nameofthm}

{%
  \begin{proof}$ $\newline
  \indent ($\implies$) 
}
{\end{proof}}

\newtheorem*{theorem*}{Theorem}
\newtheorem*{lemma*}{Lemma}
\newtheorem*{proposition*}{Proposition}

\newtheorem*{witt}{Witt's Lemma}
\newtheorem{lemma}[theorem]{Lemma}
\newtheorem{corollary}[theorem]{Corollary}
\newtheorem{proposition}[theorem]{Proposition}

\theoremstyle{definition}

\newtheorem*{definition*}{Definition}
\newtheorem{definition}[theorem]{Definition}
\newtheorem{notation}[theorem]{Notation}
\newtheorem{remark}[theorem]{Remark}
\newtheorem{example}[theorem]{Example}

\newtheorem*{notation*}{Notation}

\newtheorem{algorithm}[theorem]{Algorithm}

\newcommand{\nameofenv}{}
\newtheorem*{generic}{\nameofenv}

\newcommand{\nameOfNumEnv}{}
\newtheorem{genericNumbered}[theorem]{\nameOfNumEnv}

\newcommand{\mg}{\mathcal{M}_g}
\newcommand{\mgbar}{\overline{\mathcal{M}}_g}

\newcommand{\mgnbar}{\overline{\mathcal{M}}_{g,n}}
\newcommand{\cgnbar}{\overline{\mathcal{C}}_{g,n}}

\newcommand{\cgbar}{\overline{\mathcal{C}}_g}

\newcommand{\sg}{\mathcal{S}_g}

\newcommand{\sgbar}{\overline{\mathcal{S}}\vphantom{\mathcal{S}}_g}

\newcommand{\sgtimesbar}{\overline{\mathcal{S}}\vphantom{\mathcal{S}}^{\times m}_g}
\newcommand{\sgtimesbarm}[1]{\overline{\mathcal{S}}\vphantom{\mathcal{S}}^{\times #1}_g}

\newcommand{\sgm}{\mathcal{S}^m_g}
\newcommand{\sgmbar}{\overline{\mathcal{S}}\vphantom{\mathcal{S}}^m_g}

\newcommand{\sgt}{\sgm} 
\newcommand{\sgtm}[1]{\mathcal{S}^{#1}_g}

\newcommand{\sga}{\mathcal{S}^{\mathfrak{a}}_g}
\newcommand{\sgra}{\mathcal{S}^{(R,\mathfrak{a})}_g}

\newcommand{\csnbar}{\overline{\mathcal{S}}(\mathcal{N})}
\newcommand{\csnm}{\mathcal{S}^m(\mathcal{N})}
\newcommand{\csnmlim}{\mathcal{S}^m_{\mathrm{lim}}(\mathcal{N})}
\newcommand{\csnmm}[1]{\mathcal{S}^{#1}(\mathcal{N})}

\newcommand{\csnmbar}{\overline{\mathcal{S}}\vphantom{\mathcal{S}}^m(\mathcal{N})}
\newcommand{\csnmmbar}[1]{\overline{\mathcal{S}}\vphantom{\mathcal{S}}^{#1}(\mathcal{N})}

\newcommand{\art}{\operatorname{Art}}

\makeatletter
\def\imod#1{\allowbreak\mkern10mu({\operator@font mod}\,\,#1)}
\makeatother

\newcommand{\del}{\partial}

\newcommand{\p}{\Pp^1}

\newcommand{\inv}{^{-1}}
\newcommand{\set}[1]{\{ #1 \}}

\newcommand{\tos}{\twoheadrightarrow}
\newcommand{\toi}{\hookrightarrow}
\newcommand{\isoto}{\overset{\sim}{\to}}
\newcommand{\iso}{\overset{\sim}{\to}}

\newcommand{\Sets}{{(\text{Sets})}}
\newcommand{\hart}{\text{\^{A}rt}}

\newcommand{\Cc}{\mathbbm{C}}

\newcommand{\Ff}{\mathbbm{F}}

\newcommand{\Nn}{\mathbbm{N}}

\newcommand{\Pp}{\mathbbm{P}}

\newcommand{\Zz}{\mathbbm{Z}}

\renewcommand{\H}{\mathrm{H}}

\newcommand{\xa}{\mathfrak{a}}

\newcommand{\xF}{\mathfrak{F}}

\newcommand{\xj}{\mathfrak{j}}
\newcommand{\xo}{\mathfrak{o}}
\newcommand{\xp}{\mathfrak{p}}

\newcommand{\xL}{\mathfrak{L}}
\newcommand{\xm}{\mathfrak{m}}

\newcommand{\ca}{\mathcal{A}}

\newcommand{\cc}{\mathcal{C}}

\newcommand{\ce}{\mathcal{E}}
\newcommand{\cf}{\mathcal{F}}
\newcommand{\cg}{\mathcal{G}}

\newcommand{\cl}{\mathcal{L}}

\newcommand{\cm}{\mathcal{M}}
\newcommand{\cn}{\mathcal{N}}
\newcommand{\co}{\mathcal{O}}

\newcommand{\cR}{\mathcal{R}}

\newcommand{\cx}{\mathcal{X}}
\newcommand{\cy}{\mathcal{Y}}

\newcommand{\stacks}[1]{\cite[\href{http://stacks.math.columbia.edu/tag/#1}{Tag #1}]{stacks-project}}

\newcommand{\tphi}{\widetilde \varphi}
\newcommand{\hco}{\hat \co}

\author[E.~C.~Sert\"oz]{Emre Can Sert\"oz}
\address{Emre Can Sert\"oz\\
Leibniz University Hannover, Welfengarten 1, 30167 Hannover, Germany}
\email{emre@sertoz.com}
\title[A compactification of the moduli space of multiple-spin curves]{A compactification of the moduli space of\\ multiple-spin curves} 
\date{\today}
\subjclass[2020]{14H10, 14D23, 14B10}
\keywords{spin curves, compactification, moduli, theta-characteristics, roots of line bundles}

\usepackage{color}
\usepackage{marvosym} 

\begin{document}

\begin{abstract}
We construct a smooth Deligne--Mumford compactification for the moduli space of curves with an $m$-tuple of spin structures using line bundles on quasi-stable curves as limiting objects, as opposed to line bundles on stacky curves. For all $m$, we give a combinatorial description of the local structure of the corresponding coarse moduli spaces. We also classify all irreducible and connected components of the resulting moduli spaces of multiple-spin curves. 
\end{abstract}

\maketitle

\section{Introduction} 

The configuration of the $28$ bitangents of a smooth quadric curve in the plane form a beautiful chapter in classical algebraic geometry~\cite{riemann-g3,dolgachev}. A similar structure is observed again with the $120$ tritangents of a generic canonical space sextic~\cite{coble--theta-book,DelCentina-Recillas1983,lehavi15,lehavi22,bruin-sertoz} or, more generally, with the ${2^g \choose 2}$ contact hyperplanes of a generic canonically embedded curve of genus~$g$~\cite{caporaso-sernesi}. The unifying idea is that these hyperplanes correspond to odd spin structures on the curve, i.e., to square roots of the canonical bundle of the curve with a non-zero section.

In the study of algebraic curves, degeneration techniques play an important role. However, there is a technical gap prohibiting the study of ``configurations'', or tuples, of spin structures on curves via degeneration: there is no compactification of the relevant moduli spaces that uses curves with line bundles as limiting objects. There are notable exceptions where this gap was partially addressed, or circumvented, with great success. Most prominently, the moduli space of spin curves was compactified by Cornalba~\cite{cornalba} which gave rise to a complete Kodaira classification of the resulting spaces~\cite{farkas-even,farkas-verra-even}. These moduli spaces allow only for the study of a \emph{single} spin structure on each curve. On the other extreme, Caporaso and Sernesi~\cite{caporaso-sernesi} considered the degeneration of \emph{all} odd spin structures to prove that a generic curve is determined by its contact hyperplanes. The work~\cite{caporaso-sernesi} could largely avoid the aforementioned technical gap because one can degenerate curves together with all of their odd or even spin structures without constructing a new moduli space. Fan, Jarvis and Ruan~\cite{fan-jarvis-ruan} do indeed construct a compactification of the moduli space of curves with more than one spin structure using \emph{stacky} curves as limiting objects. 

We find that the quasi-stable curves of Cornalba retain their intuitive appeal in approaching projective geometric problems. In particular, the study of effective limit linear systems keeps their geometric flavor when working with quasi-stable curves. For this reason, we develop here a compactification of the moduli spaces of ``multiple spin curves'' from the point of view of quasi-stable curves. This construction opens the door to studying multiple and fractional limit linear series on quasi-stable curves. We establish local structural properties of these moduli spaces, such as the smoothness of the moduli stacks and the quotient singularity types of the coarse moduli spaces. We also classify the connected components of these spaces. We describe these results in greater detail in Section~\ref{sec:idea}.  

\subsection{Idea of the construction and main results}\label{sec:idea}

Let $k$ be an algebraically closed field of characteristic not two. A \emph{spin structure} on a proper smooth curve $C/k$ is a pair $(L,\alpha\colon L^{\otimes 2} \isoto \omega_{C/k})$ where $L$ is a line bundle on $C$ and $\omega_{C/k}$ is the canonical bundle of $C$. The tuple $(C,L,\alpha)$ is called a \emph{spin curve}. The moduli space $\sg$ of spin curves of genus $g$ is quasi-finite over the moduli space $\mg$ of curves of genus $g$.

Fix a positive integer $m$. A sequence $(L_i, \alpha_i)_{i=1}^m$ of spin structures on $C$ will be called a \emph{multiple-spin structure on $C$}. The tuple $(C,(L_i,\alpha_i)_{i=1}^m)$ will be called a \emph{multiple-spin curve}. For each genus $g$, the moduli space of multiple-spin curves will be denoted by $\sgm$ which is the $m$-fold fiber product $\sg \times_{\mg} \dots \times_{\mg} \sg$. The purpose of this paper is to give a natural compactification of $\sgm$, determine its irreducible components, and describe its basic geometric properties. 

There is a compactification $\sgbar \to \mgbar$ of $\sg$ over the moduli space $\mgbar$ of stable curves, whose coarse moduli scheme over the complex numbers was originally constructed by Cornalba~\cite{cornalba}. Later, the work was completed by Jarvis~\cite{jarvis-torsion-free} by the construction of the moduli stack.

Cornalba compactifies $\sg$ by considering \emph{limit spin curves}, which are tuples of the form $(X,\cl,\alpha\colon \cl^{\otimes 2} \to \omega_{X/k})$ where: 
\begin{itemize}
  \item $X$ is a \emph{quasi-stable} curve of genus $g$,
  \item $\cl$ is a line bundle of degree $g-1$ on $X$ and degree 1 on unstable components,
  \item $\alpha$ is an isomorphism away from the unstable components and zero on the unstable components. 
\end{itemize} 
Note that any unstable component of $X$ is isomorphic to $\p$ and such a component contains exactly two nodes of $X$. The stabilization of $X$ contracts each unstable component to a point.

Our compactification will be based upon that of Cornalba's and, in particular, on the product $\sgtimesbar \colonequals  \sgbar \times_{\mgbar} \dots \times_{\mgbar} \sgbar$. Although the product $\sgtimesbar$ is compact, the objects it parametrizes are not entirely natural. By definition, an element $(\pi_i\colon X_i \to C, \cl_i, \alpha_i)_{i=1}^{m} \in \sgtimesbar(k)$ is based upon a tuple of curves $(X_i)_{i=1}^m$ whose stabilizations $C$ are all identified but the unstable components of $X_i$'s remain distinct. Not only is it unnatural to work with $m$ partially identified curves, but this causes the space $\sgtimesbar$ to be non-normal. We prove in Proposition~\ref{prop:normalize} that the compactification we give normalizes $\sgtimesbar$.

Roughly, the normalization of $\sgtimesbar$ could be constructed by adding to each tuple $(\pi_i\colon X_i \to C, \cl_i, \alpha_i)_{i=1}^{m} \in \sgtimesbar$ a master curve $X$ to dominate all $X_i$. For each $i=1,\dots,m$ let $\chi_i \subset C$ denote the subset of the nodes over which $\pi_i\colon X_i \to C$ is not an isomorphism. We define a ``destabilization'' $\pi \colon  X \to C$ of $\chi = \cup_{i=1}^m \chi_i$ by inserting a rational curve to separate the branches at each node in $\chi$ (see Definition~\ref{def:blow-up}). For each $i=1,\dots,m$ we can factor the map $\pi$ into $\pi_i \circ \rho_i$ using a partial stabilization map $\rho_i \colon  X \to X_i$.
Denote by $(\cl_i',\alpha_i')$ the pullback of $(\cl_i,\alpha_i)$ via $\rho_i$. At this point the intermediate objects $(\pi,\cl_i)$ can be forgotten and we may consider the tuple $(X,(\cl_i',\alpha_i')_{i=1}^m)$ as a natural limit of multiple-spin curves.

We must warn, however, that this construction needs to be refined as there are typically infinitely many choices of maps $(\rho_i)_{i=1}^m$ leading to infinitely many non-equivalent tuples $(\cl_i',\alpha_i')_{i=1}^m$, even modulo automorphisms of $X$. In other words, the resulting moduli space of such tuples does not have finite fibers over $\mgbar$.

\begin{figure}[h]
  \centering
  \begin{tikzcd}
                   & X   \arrow[dl,"\rho_1"'] \arrow[d, "\rho_2"] \arrow[drr,"\rho_m"] &       & \\
    X_1 \arrow[dr,"\pi_1"'] & X_2 \arrow[d,"\pi_2"]                        & \dots & X_m \arrow[dll,"\pi_m"]\\
                   & C                                    &       &
  \end{tikzcd}
  \caption{The unstable components in the master curve $X$ fully contract to give $C$. Contracting various subsets of all unstable components of $X$ we can obtain the $X_i$'s.}
  \label{fig:master_curve}
\end{figure}

\begin{figure}[h]
  \centering
  \includegraphics[width=0.5\textwidth]{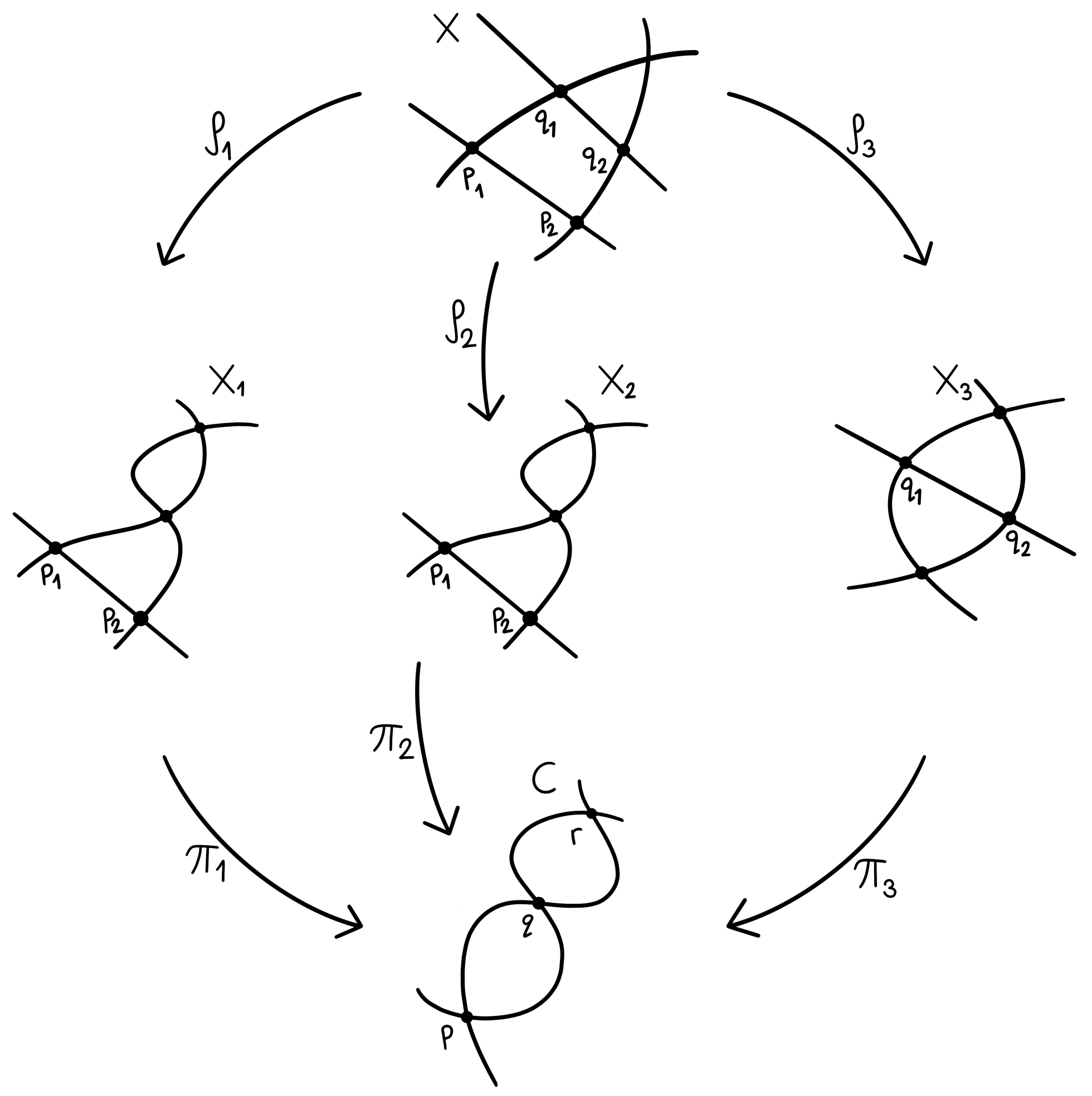}
  \caption{An example with $m=3$ illustrating the construction in Figure~\ref{fig:master_curve}. The base curve $C$ consists of two smooth components meeting to form three nodes $p,q,r$. The curves $X_1$ and $X_2$ each contract an unstable component $\p$ over $p$, whereas $X_3$ contracts a $\p$ over $q$. The master curve $X$ contracts a $\p$ over $p$ and $q$ each. The morphisms $\rho_1$ and $\rho_2$ from $X$ serve to identify the $\p$'s in $X_1$ and $X_2$ over $p$.} 
  \label{fig:example}
\end{figure}

We fix the issues of the rough construction by imposing a constraint on the choice of $\rho_i$. Infinitely many non-equivalent choices of $\rho_i$'s occur only when there is a pair of indices $i\neq j$ for which $X_i$ and $X_j$ have unstable components over the same node of $C$, i.e., when $\chi_i \cap \chi_j \neq \emptyset$. In this case, the line bundles $\cl_i$ and $\cl_j$ both have degree one on the unstable components lying above $\chi_i \cap \chi_j$ and they can be used to stabilize these components. 

We will, then, allow only for sequences of maps $(\rho_i)_{i=1}^m$ where for every $i,j$ and every $x \in \chi_i \cap \chi_j$ the line bundles $\rho_i^* \cl_i^{\otimes 2}$ and $\rho_j^*\cl_j^{\otimes 2}$ are isomorphic in an open neighbourhood of $\pi\inv(x) \subset X$. The resulting tuple $(X,(\cl_i',\alpha_i')_{i=1}^m)$ will be called a \emph{limit multiple-spin curve}. See Section~\ref{sec:multiple_spin} for examples and an intrinsic definition which encapsulates \emph{families} of limit multiple-spin curves. We denote the moduli space of limit multiple-spin curves by $\sgmbar$.

\begin{theorem}\label{thm:intro_compactify}
  The moduli space of limit multiple-spin curves $\sgmbar$ is a smooth and proper Deligne--Mumford stack. Moreover, the canonical inclusion $\sgm \toi \sgmbar$ is dense and open. The forgetful map $\sgmbar \to \mgbar$ is quasi-finite.
\end{theorem}
\begin{remark}
  This is a special case of our Theorem~\ref{thm:main}. Another special case gives a similar compactification for the moduli space of marked curves with roots of the twisted canonical bundle. Section~\ref{sec:proof_main} explains these deductions.
\end{remark}

For the birational classification of moduli spaces, understanding the nature of the singularities of the coarse moduli spaces associated to the moduli stacks is often a necessary step~\cite{harris-mumford--kodaira,ludwig,farkas-ludwig,chiodo-farkas}. The singularities of the coarse moduli space associated to $\sgmbar$ are finite quotient singularities. These quotients are described by the action of the automorphism groups of multiple-spin curves on their local deformation functors. In Section~\ref{sec:aut_groups} we give a completely combinatorial description of these group actions using a form of dual graph associated to each curve, see Theorem~\ref{thm:aut_action} and Proposition~\ref{prop:combinatorial_aut}.

\medskip

A complete classification of the connected components of $\sgmbar$ is possible. Note that $\sgmbar$ is smooth and, therefore, its irreducible and connected components coincide. Moreover, $\sgm$ is Zariski dense in $\sgmbar$ so the irreducible and connected components of these two spaces coincide. 

Given an $m$-spin structure $(L_i,\alpha_i)_{i=1}^m$ on a genus $g$ curve $C$ we call the associated syzygy relations to be the tuple 
\begin{equation}
\left( h^0(L_i) \mod 2;\,\, h^0(L_j\otimes L_k \otimes L_1^\vee) \mod 2 \right)_{\substack{1 \le i \le m \\ 1 < j < k \le m}} \in \Ff_2^{m + {m-1 \choose 2}}. 
\end{equation}
Any tuple that can be obtained in this manner will be called an $(m,g)$-\emph{syzygy relation}. Let $\gr(a,b)$ be the Grassmannian of $a$-planes in $\Ff_2^b$. Section~\ref{sec:components} contains the proof of the following theorem.

\begin{theorem}\label{thm:intro_components}
  The irreducible (and connected) components of $\sgmbar$ are in natural bijection with tuples $(W,\xa)$ where $W \in \gr(m-k,m)$, $\xa \in \Ff_2^{k+{k-1 \choose 2}}$ is a $(k,g)$-syzygy relation, and $k=0,\dots,m-1$. If $g \ge k$ then every $\xa\in \Ff_2^{k+{k-1 \choose 2}}$ is a $(k,g)$-syzygy relation and if $g < k$ then Algorithm~\ref{alg:syzygy_relations} finds all $(k,g)$-syzygy relations.
\end{theorem}

\subsection{Outline of the paper}

In Section~\ref{sec:limit_roots} we expand the notion of a limit multiple-spin curve defined above and give rigorous definitions. These definitions, employing line bundles on quasi-stable curves, are convenient for solving geometric problems but not for solving the present problem of studying the structure of the relevant moduli spaces. Therefore, in Section~\ref{sec:torsion_free_sheaves} we give an equivalent definition using torsion-free sheaves on stable curves. In Section~\ref{sec:moduli_is_algebraic} we prove that the moduli problem is represented by an algebraic stack. In Section~\ref{sec:local_analysis} we undertake a study of the local structure of the moduli stacks. In Section~\ref{sec:fundamental_properties} we prove basic geometric properties of these moduli stacks. In Section~\ref{sec:components} we classify the components of $\sgm$. In Appendix~\ref{app:local_deformations} we establish some of the technical aspects required for the study of the local deformation functors of limit multiple-spin curves.

\subsection{Acknowledgments}

It is my pleasure to thank my adviser Gavril Farkas for generously sharing his insight into research as well as giving me financial and academic support during the course of my PhD. I would like to thank my co-adviser Gerard van der Geer for numerous discussions during my stay in Amsterdam. In addition, thanks to Lenny Taelman and David Holmes for providing helpful suggestions at key moments. Special thanks go to Fabio Tonini for helping me with stacks and to Klaus Altmann for helping me extend my scholarship. Finally, I thank Özde Bayer Sertöz for the help with the picture. This research constitutes a chapter in my PhD thesis. My PhD was funded by the Berlin Mathematical School and Graduiertenkolleg 1800 of the Deutsche Forschungsgemeinschaft. I am grateful to the referee for their careful reading and insightful comments. 

\section{Families of multiple limit roots}\label{sec:limit_roots}

In this section we generalize the construction of limit multiple spin curves given in the introduction to families of curves. As there is nothing special about square roots of the canonical bundle from the point of view of our construction, we will consider the square roots of any line bundle. 

\subsection{Destabilization of curves}

Let $k$ be an algebraically closed field with $\chr k \neq 2$. Let $C$ be a connected nodal curve over $k$. If the relative dualizing sheaf of $C$ is ample then $C$ is said to be \emph{stable}.  Let $\chi \in C$ be a subset  of the nodes of $C$ and $I \subset \co_C$ the ideal sheaf corresponding to $\chi$.  

\begin{definition}\label{def:blow-up}
  Let  $\pi\colon  X = \Proj_C (\sym^* I) \to C$. Then the map $\pi$, and sometimes $X$, is called a \emph{destabilization of $C$ at $\chi$}. If $C$ is stable then $X$ is \emph{quasi-stable}.
\end{definition}

\noindent  For each $x \in \chi$, the fiber $\pi\inv(x)$ is isomorphic to $\p_k$. The map $\pi$ is an isomorphism over $C \setminus \chi$ .

\begin{definition}
  Let $\pi\colon X \to C$ be a destabilization of the nodes $\chi \subset C$. Then for each $x \in \chi$ we will call the fiber $\pi\inv(x) \subset X$ an \emph{exceptional component} of $X$.
\end{definition}

\subsection{Families of limit roots}

We recall the notion of a limit root given in Definition~2.1.1 of~\cite{caporaso-casagrande-cornalba} but only for \emph{square} roots. All the definitions in this subsection are adapted from \emph{loc.\ cit.} Fix a line bundle $N$ on the curve $C$ of even degree. Consider a triplet $(\pi\colon  X \to C,L,\alpha\colon  L^{\otimes 2} \to \pi^*N)$ where $\pi$ is a destabilization and $L$ is a line bundle on $X$ of degree $(\deg N)/2$.

\begin{definition}\label{def:limit-root}
  Suppose $\alpha$ is an isomorphism in the complement of the exceptional components of $X$. If $L$ has degree~$1$ on each exceptional component then $(\pi,L,\alpha)$ is a \emph{limit root of $N$}. If $L$ has degree~$0$ or $1$ on the exceptional components then $(\pi,L,\alpha)$ \emph{stabilizes to a limit root of~$N$}.
\end{definition}

Let $T$ be a scheme over $\Zz[\frac{1}{2}]$, $\cc/T$ be a family of stable curves~\stacks{0E73} and $\cx/T$ a family of nodal curves~\cite[\href{http://stacks.math.columbia.edu/tag/0C58}{Tag 0C58}]{stacks-project}. Fix a line bundle $\cn$ on $\cc$ of relative even degree $d$.

\begin{definition}\label{def:family-of-blow-ups}
  If a morphism $\pi\colon \cx \to \cc$ restricts on each geometric fiber to a destabilization then $\pi$ is a \emph{destabilization}.
\end{definition}

\noindent Let $\pi\colon \cx \to \cc$ be a destabilization, $\cl$ a line bundle on $\cx$ and $\alpha\colon \cl^{\otimes 2}\to \pi^*\cn$ a morphism.

\begin{definition}
   If at each geometric fiber of $(\cx \overset{\pi}{\to} \cc,\cl,\alpha)$ is (or stabilizes to) a limit root then $(\cx \overset{\pi}{\to} \cc,\cl,\alpha)$ is (or stabilizes to) a \emph{family of limit roots}. 
\end{definition}

If $(\pi,\cl,\alpha)$ stabilizes to a family of limit roots then there exists a family of limit roots $(\pi'\colon \cx' \to \cc,\cl',\alpha')$ and a morphism $\rho\colon \cx \to \cx'$ such that $(\cl,\alpha)\simeq \rho^*(\cl',\alpha')$. The map $\rho$ is the \emph{partial stabilization map with respect to $\cl$}. A partial stabilization contracts the unstable components of each fiber on which $\cl$ has degree~$0$.

\begin{notation}\label{not:V}
  If $(\pi,\cl,\alpha)$ stabilizes to a family of limit roots then denote by $V(\cl) \subset \cx$ the largest open set on which the partial stabilization $\rho$ is an isomorphism. 
\end{notation}

An isomorphism between two families of limit roots $(\cx_i \overset{\pi_i}{\to} \cc,\cl_i,\alpha_i)_{i=1,2}$ is a pair of isomorphisms $(f\colon \cx_1 \to \cx_2,g\colon \cl_1 \to f^*\cl_2)$ such that $\alpha_1=f^*\alpha_2 \circ g^{\otimes 2}$.

\subsection{Families of multiple limit roots}\label{sec:families_of_multiple_limit_roots}

We now want to consider $m$-tuples of limit roots of $\cn$ over $\cc$ for a fixed positive integer $m$.

\begin{definition} \label{def:multiple-limit-root}
  Let $\pi\colon \cx \to \cc$ be a destabilization. Let $\xL\colonequals \set{\cl_i,\alpha_i\colon \cl_i^{\otimes 2} \to \pi^*\cn}_{i=1}^m$ be such that each $(\cl_i,\alpha_i)$ stabilizes to a limit root, but $\xL$ itself is not pulled back from a partial stabilization. Consider a line bundle $\cl$ and a sequence of morphisms $\varphi_i\colon \cl \to \cl_i^{\otimes 2}$ satisfying the following:
  \begin{itemize}
    \item $\alpha_i \circ \varphi_i = \alpha_j \circ \varphi_j$ for each $i,j$.
    \item Each $\varphi_i$ restricts to an isomorphism on $V(\cl_i)$, see Notation~\ref{not:V}.
  \end{itemize}
  Then, we will call $\xF=(\varphi_i)_{i=1}^2$ a \emph{synchronization data}. The tuple $(\pi,\xL,\xF)$ will be called a \emph{multiple limit root}. An isomorphism of multiple limit roots is a sequence of isomorphism of the limit roots commuting with the synchronization data.
\end{definition}

\begin{remark}
  The definition above agrees with the construction given in the introduction. Indeed, on $V(\cl_i) \cap V(\cl_j)$ the map $\varphi_j \circ \varphi_i\inv$ identifies $\cl_i^{\otimes 2}$ and $\cl_j ^{\otimes 2}$.
\end{remark}

We will give an alternative formulation of multiple limit roots in the next section and prove the equivalence of these two notions in Proposition~\ref{prop:equivalence}. To make this equivalence precise, we make the following definition.

\begin{definition}\label{def:cat_limit}
  For $(C \to T, \cn)$ and $m$ fixed, let $\csnmlim$ denote the \emph{category of limit multiple roots of $\cn$}. That is, $\csnmlim \to T$ is fibered in groupoids with fiber over $T' \to T$ consisting of the multiple limit roots of $(C|_{T'} \to T', \cn|_{T'})$.
\end{definition}

\subsection{Multiple-spin curves}\label{sec:multiple_spin}

Take $N$ to be the canonical bundle. In this case, limit roots are called limit spin curves, therefore we will refer to multiple limit roots as a multiple-spin curves. Taking $m=2$ we will give some examples of multiple-spin curves. For basic results on limit spin curves we refer to~\cite{cornalba}. We will say that a spin curve $\xi=(\pi\colon \tilde C \to C,L,\alpha)$ over $C$ \emph{requires destabilization} if $\pi$ is not an isomorphism.

The examples above can be readily computed from definitions. Especially the Appendix~\ref{app:local_deformations} and Section~\ref{sec:local_analysis} are helpful in constructing more elaborate examples. A form of dual graph is introduced in Section~\ref{sec:aut_groups} which gives the general framework to tackle examples like the ones below. For a detailed study of the examples below and additional examples we refer to~\cite[\S~II.1.6.2]{me-thesis}.

\begin{example}\label{ex:red}
  For $i=1,2$ let $(C_i,p_i)$ be marked smooth irreducible curves. Consider the nodal curve $C=C_1\cup_{p_1 \sim p_2 } C_2$. Every spin curve over $C$ requires destabilization. Given any pair of spin curves $\xi_i=(\pi_i,L_i,\alpha_i)$, $i=1,2$, we can \emph{uniquely} form a limit multiple-spin curve over $(\xi_1,\xi_2)$ upto isomorphisms.
\end{example}

\begin{example}\label{ex:irred}
  Let $C = \overline{C}/(p\sim q)$ where $p,q \in \overline{C}$ are distinct points on a smooth irreducible curve $\overline{C}$. Some spin curves over $C$ requires destabilization and some do not. If we take a pair of spin curves $\xi_i=(\pi_i,L_i,\alpha_i)$, $i=1,2$, over $C$ then there is a unique way to form a multiple spin curve, upto isomorphisms, \emph{unless} both spin curves require destabilization. In the latter case, from each pair $(\xi_1,\xi_2)$ we can form two distinct isomorphism classes of multiple-spin curves. 
\end{example}

There are three kinds of interactions that can happen over each node between two limit spin structures: both, one or neither may require destabilization. The following demonstrates all three on one curve. 

\begin{example}\label{ex:three_nodes}
  Consider a smooth irreducible curve $\overline{C}$ with six distinct points $p_{ij} \in \overline{C}$ for $i=1,2,3$ and $j=1,2$. Form the curve $C=\overline{C}/(p_{i1} \sim p_{i2})_{i=1}^{3}$ and let $x_i$ stand for the node $[p_{i1}]$. We will take limit roots $\xi_1$ and $\xi_2$ over $C$ such that $\xi_1$ requires only the node $x_1$ to be blown-up while $\xi_2$ requires $x_1$ and $x_2$ to be blown-up. We can check that there are two isomorphism classes of limit multiple-spin curves over $(\xi_1,\xi_2)$.
\end{example}

\section{Torsion-free sheaves}\label{sec:torsion_free_sheaves}

The definition given in Section~\ref{sec:limit_roots} is what is intended to be used in geometric applications. However, for the problem of constructing the relevant moduli spaces and studying the local deformation spaces, we gain an advantage by working with stable curves instead of quasi-stable curves. We will now give our ``working definition'' for limit multiple roots using torsion-free sheaves.

\subsection{Torsion-free roots}  Let $\cm$ be an algebraic stack. The definitions in this subsection are from~\cite{jarvis-torsion-free}.

\begin{definition}[Jarvis]
  A \emph{torsion-free sheaf on a stable curve $\cc \to \cm$} is a coherent $\co_{\cc}$-module $\ce$ which is flat and of finite presentation over $\cm$ such that over each $s \in \cm$ the fiber $\ce|_{\cc_s}$ has no associated primes of height one. 
\end{definition}

\noindent Note that the smooth locus of the map $\cc\to\cm$ is contained in the locus where $\ce$ is locally free. 

\begin{definition}[Deligne, Jarvis]\label{def:del-jar}
  Let $\ce$ be a rank-1 torsion-free sheaf on a curve $\cc \to \cm$ and $\cn$ a line bundle on $\cc$. Let $\delta \colon  \ce \iso \cn \otimes \ce^\vee$ be an isomorphism. Then the pair $(\ce,\delta)$ will be called a \emph{(square) root of $\cn$}.
\end{definition}

\begin{definition}\label{def:bilinear}
  Given a coherent module $\ce$ and a line bundle $\cn$ on a scheme $\cx$, a homomorphism $b \colon  \ce^{\otimes 2} \to \cn$ will be called a \emph{bilinear form}. A bilinear form induces two maps $b^l,b^r \colon  \ce \to \ce^\vee \otimes \cn$ where $\ce^\vee = \hom(\ce,\co_{\cx})$, $b^r(e)=b(e,\_)$ and $b^l(e) = b(\_,e)$. 
 If both $b^r$ and $b^l$ are isomorphisms then $b$ is \emph{non-degenerate}. If $b^r=b^l$ then $b$ is \emph{symmetric}, and $b$ factors through the symmetrizing map $\ce^{\otimes 2} \to \sym^2 \ce$. 
\end{definition}

We will adopt an unusual notational custom and for any $A$-module $E$ write the $d$-th symmetric product $\sym^d_A(E)$ simply as $E^d$, and given $\mu \colon E \to F$ we will denote by $\mu^d$ the induced map $E^d \to F^d$. The same goes for sheaves of modules and morphisms between them. In compensation, we will write out tensor powers and direct sums explicitly as $E^{\otimes d}$ and $E^{\oplus d}$, respectively. We will use the following reformulation of Definition~\ref{def:del-jar}.

\begin{definition}\label{def:square-root}
  Let $\ce$ be a rank-1 torsion-free sheaf on a curve $\cc \to \cm$ and $\cn$ a line bundle on $\cc$. Let $b \colon  \ce^{2} \to \cn$ be a non-degenerate symmetric form. Then the pair $(\ce,b)$ will be called a \emph{(torsion-free) root of $\cn$ on $\cc/\cm$}.
\end{definition}

  An isomorphism $\mu \colon (\ce,b) \to (\ce',b')$ of roots is defined to be an isomorphism of the underlying sheaf of modules $\mu \colon \ce \iso \ce'$ such that $b=b' \circ \mu^{2}$.

\subsection{Relation to limit roots}\label{sec:torsion_to_limit_roots}

Suppose $(\ce,b)$ is a torsion-free root of $\cn$ on $\cc/\cm$. Define $\Pp(\ce)\colonequals \Proj_\cc(\sym^* \ce)$ with $\pi\colon \Pp(\ce)\to \cc$ the structure map and $\cl = \co_{\Pp(\ce)}(1)$ the line bundle corresponding to the $\Proj$ construction. Notice that $\pi$ is a destabilization. 

There are natural surjective maps $\pi^* \ce^d \to \cl^{\otimes d} = \co_{\Pp(\ce)}(d)$ for each $d \ge 0$ and there is the map $\pi^*b \colon \pi^*\ce^2 \to \pi^* \cn$. As is shown in \S 3.1.3 of~\cite{jarvis-torsion-free} there is a natural map $\alpha$ making the following diagram commute:
\[
  \begin{tikzcd}
    \pi^*\ce^2 \arrow[rd,"\pi^*b"]\arrow[d] & \\
    \cl^2 \arrow[r,dotted,"\alpha"]         & \pi^* \cn
  \end{tikzcd}.
\]

\begin{proposition}\label{prop:pull}
  Let $(\ce,b)$ be a root of $\cn$. Then $(\Pp(\ce) \overset{\pi}{\to} \cc,\co_{\Pp(\ce)}(1),\alpha)$, constructed above, is a family of limit roots of $\cn$.
\end{proposition}
\begin{proof}
  Both $\Proj$ and $\sym$ constructions behave well with respect to base change. So we may reduce to $T=\spec k$ where $k$ is an algebraically closed field.  Let $L=\co(1)$ and note $\pi_* L \simeq \ce$, see Lemma 3.1.4.(2)~\cite{jarvis-torsion-free}. To see that $L$ has degree one over any exceptional fiber $E$ over a node $x$ we simply observe $h^0(E,L|_E) = \dim_{k} \ce|_x = 2$. Since $E \simeq \p_k$ we are done.  The map $\alpha$ is an isomorphism away from the exceptional divisors because $b$ is an isomorphism away from the corresponding nodes. The degrees of $L$ and $\ce$ agree because $\pi_* L \simeq \ce$. This completes the proof.
\end{proof}

Conversely, given a family of limit roots $(\pi \colon\cx \to \cc,\cl,\alpha)$, let $\ce\colonequals  \pi_*\cl$. Then, using Lemma~3.1.4.(2)~\cite{jarvis-torsion-free} again, we have $\pi_* \cl^2 \simeq \ce^2$. Using the adjunction map $a \colon \pi_*\pi^* \cn \to \cn$ we may define $b\colonequals a \circ \pi_* \alpha \colon  \ce^2 \to \pi_*\pi^* \cn \to \cn$.  

\begin{proposition}\label{prop:push}
  The tuple $(\ce,b)$ obtained in this way is a torsion-free root of $\cn$.
\end{proposition}
\begin{proof}
  This is similar to the proposition above. The main ingredients are Proposition 3.1.2.(3) and Proposition 3.1.5 of~\cite{jarvis-torsion-free} which says that $\pi_* \cl$ is torsion-free and $b$ is of the right form respectively.
\end{proof}

As in Definition~\ref{def:cat_limit} we may define the category of torsion-free roots. The results of this section imply the following.

\begin{corollary}\label{cor:categories_are_equivalent}
  The category of torsion-free roots of $\cn$ is equivalent to the category of limit roots of $\cn$.
\end{corollary}

\subsection{Multiple torsion-free roots}

Let $(\ce_i,b_i)_{i=1}^m$ be a sequence of torsion-free roots of $\cn$ on a family of curves $\cc \to \cm$. Now we will phrase the notion of multiple limit roots from Section~\ref{sec:families_of_multiple_limit_roots} in terms of torsion-free roots using the equivalence identified in Section~\ref{sec:torsion_to_limit_roots}. 

For each root we have the destabilization $\cx_i \colonequals \Proj_\cc(\sym^* \ce_i) \to \cc$, and the goal is to find a common destabilization $\cx \to \cc$. The roots are partially identified with one another after squaring them and $\cx_i$ can also be constructed by taking even powers in the symmetric algebra, i.e., $\cx_i \simeq \Proj_\cc(\sym^{2*} \ce_i)$. Our goal is then to isolate the conditions for which the even symmetric algebras are identified wherever possible.

\begin{notation}
  Let $V_i \subset \cc$ be the open locus (Lemma~\ref{lem:V_is_open}) where the rank of $\ce_i$ is maximal amongst all $\ce_j$. This is the analogue of Notation~\ref{not:V}.
\end{notation}

\begin{definition}\label{def:pre-sync}
  Let $\cf$ be a sheaf of modules on $\cc$ and let $(\varphi_i \colon \cf \to \ce_i^{2})_{i=1}^m$ be a sequence of maps such that: (1) for all $i,j$ we have $b_i\circ\varphi_i = b_j\circ\varphi_j$, (2) for all $i$ the map $\varphi_i|_{V_i}$ is an isomorphism.
  Then, we will call $(\varphi_i)_{i=1}^m$ a \emph{pre-sync data} for the sequence of roots $(\ce_i,b_i)_{i=1}^m$. The two conditions above are \emph{pre-sync conditions}.
\end{definition}

Suppose $(\varphi_i)_{i=1}^m$ is a pre-sync data for $(\ce_i,b_i)_{i=1}^m$. On $V_{ij}\colonequals  V_i \cap V_j$ we can define $\psi_{ij} = \varphi_j|_{V_{ij}} \circ \varphi_i|_{V_{ij}}\inv \colon \ce_i^2|_{V_{ij}} \iso \ce_j^2|_{V_{ij}}$. Using $\psi_{ij}$, we get a surjective map 
\begin{equation}
  \sym^* \sym^2 \ce_i|_{V_{ij}} \to \sym^{2*} \ce_j|_{V_{ij}}.
  \label{eq:sym2_descent}
\end{equation}

\begin{lemma}\label{lem:sym2_descent}
  The map (\ref{eq:sym2_descent}) factors through an isomorphism $\sym^{2*} \ce_i|_{V_{ij}} \to \sym^{2*} \ce_j|_{V_{ij}}$ if and only if $\sym^2 \sym^2 \ce_i|_{V_{ij}} \to \sym^4 \ce_j|_{V_{ij}}$ factors through $\sym^4 \ce_i|_{V_{ij}} \to \sym^4 \ce_j|_{V_{ij}}$.
\end{lemma}
\begin{proof}
  The kernel of $\sym^* \sym^2 \ce_i \to \sym^{2*} \ce_i$ is generated by the kernel of $\sym^2 \sym^2 \ce_i \to \sym^4 \ce_i$.
\end{proof}

\begin{definition}\label{def:sync}
  If the map $\sym^2 \sym^2 \ce_i|_{V_{ij}} \to \sym^4 \ce_j|_{V_{ij}}$ factors through $\sym^4 \ce_i|_{V_{ij}} \to \sym^4 \ce_j|_{V_{ij}}$ then we will call $(\varphi_i)_{i=1}^m$ a \emph{sync data} for the sequence of roots $(\ce_i,b_i)_{i=1}^m$. This condition will be called the \emph{sync condition}. 
\end{definition}

\begin{definition}\label{def:multiple-root}
  The tuple $(\ce_i,b_i,\varphi_i)_{i=1}^m$ will be called a \emph{multiple-root of $\cn$} if $(\varphi_i)_{i=1}^m$ satisfy the sync condition.
\end{definition}

An isomorphism between a pair of multiple-roots is a sequence of isomorphisms between the underlying roots compatible with the sync data.

\begin{definition}\label{def:cat_torsion}
  For $(\cc \to \cm, \cn)$ and $m$ fixed, let $\csnmbar$ denote the \emph{category of multiple-roots of $\cn$}. That is, $\csnmbar \to \cm$ is fibered in groupoids with fiber over a scheme $T/\cm$ consisting of the multiple-roots of $(\cc|_{T} \to T, \cn|_{T})$. The sub-category $\csnm \toi \csnmbar$ is defined by taking the objects where all the roots are line bundles. 
\end{definition}

\begin{proposition}\label{prop:equivalence}
  The categories $\csnmlim$ and $\csnmbar$ are equivalent.
\end{proposition}
\begin{proof} In light of Proposition~\ref{prop:push} the pushforward of a multiple limit root is clearly a multiple-root with the synchronization data of the former mapping down to the sync data of the latter.

  Conversely, let $(\ce_i,b_i,\varphi_i)_{i=1}^m$ be a multiple-root of $\cn$ on $\cc$. As in Proposition~\ref{prop:pull} let $\cx_i \to \cc$ be the destabilization obtained by the symmetric algebra of $\ce_i$ and let $(\cl_i,\alpha_i)$ be the limit root on $\cx_i$.

  Let $D_i = \sym^{2*} \ce_i$ and construct $D$ by gluing $D_i$ on the charts $V_i$ using the sync data $(\varphi_i)_{i=1}^m$, this construction equips $D$ with a map to $\sym^{2*} \cn$. Since $D_i$ are isomorphic to $\sym^{2*} \cn$ on an open neighbourhood containing the complement of $V_i$, we can construct maps $D \to D_i$.

  Let $\cx \to \cc$ be obtained as $\Proj_\cc(D)$ with $\cl\colonequals  \co_\cx(1)$. Since $\cx_i$ can also be constructed as $\Proj_\cc(D_i)$, the maps $D \to D_i$ induce partial contraction maps $\rho_i\colon \cx \to \cx_i$ and maps $\cl \to \rho_i^*\cl_i^{\otimes 2}$.  It is now clear that the tuple $(\cx/\cc,\rho_i^*\cl_i,\rho_i^*\alpha_i,\cl \to \rho_i^*\cl_i^{\otimes 2})$ is a multiple limit root.

  These two constructions are inverses to one another and they are functorial with respect to pullback because pushforward and $\Proj$ are functorial.
\end{proof}

\begin{remark}
  Examples of multiple roots, parallel to the examples given in~Section~\ref{sec:multiple_spin} can be found in~\cite[\S~1.6.1]{me-thesis}.
\end{remark}

\subsection{Bounded degree} For technical reasons we need to introduce a boundedness condition on the degree of the line bundles. For the rest of this article we assume our line bundles have absolutely bounded degree in the following sense.

\begin{definition}\label{def:bounded-degree}
  If there exists a constant $c \in \Zz$ such that on any component $Y$ of any geometric fiber of $\cc \to \cm$ we have $\deg \cn|_Y \ge c$ then $\cn$ will be said to have \emph{absolutely bounded degree}.
\end{definition}

This boundedness condition is weak enough that unless $\cm$ has infinitely many disconnected components, the condition is automatically satisfied. Even without this condition, the line bundle $\omega_{\cc/\cm}^{\otimes l}$ for any $l \in \Zz$ has absolutely bounded degree (Sublemma 4.1.10 \cite{jarvis-torsion-free} for $l=1$, the idea readily generalizes to all $l \in \Zz$).

\section{The moduli space of multiple-roots is algebraic}\label{sec:moduli_is_algebraic}

Let $m$ be a positive integer $m$. In this section we will show that the moduli space $\csnmbar$ (Definition~\ref{def:cat_torsion}) is an algebraic stack and is locally of finite type over the base $\cm$. Note that we assume $\cn$ has absolutely bounded degree (Definition~\ref{def:bounded-degree}). 

\subsection{The moduli space of single roots}\label{sec:single_roots}

When $m=1$ we call $\csnmbar$ the moduli space of single roots and denote it by $\csnbar$. The basic properties of $\csnbar$ follow directly from the work of Jarvis~\cite{jarvis-torsion-free}. We will list these properties, briefly highlighting the differences in proofs.

\begin{theorem}\label{thm:m_is_1}
  The category $\csnbar$ is an algebraic stack. Moreover, the morphism $\pi \colon\csnmm{} \to \cm$ is proper, of finite type and quasi-finite. The diagonal of this morphism is finite and unramified. If $\cm$ is a Deligne--Mumford stack then so is $\csnmm{}$.
\end{theorem}

\begin{proof}
  When $\cn = \omega_{\cc/\cm}$ these results from those of~\cite{jarvis-torsion-free}. However, those proofs apply with little modification to the present case. The only condition required is that of absolutely bounded degree (Definition~\ref{def:bounded-degree}) which we assume. The last statement follows from the condition on the diagonal, see also~\stacks{04YV}. 
\end{proof}

\subsection{The moduli space of multiple-roots}\label{sec:multiple_moduli}

\begin{lemma}\label{lem:forget}
  For each proper subset $J \subsetneq \{1,\dots,m\}$ there is a natural forgetful functor $\csnmbar \to \csnmmbar{\# J}$ obtained by forgetting all the roots except for those at position $i \in J$ and adjusting the sync data appropriately.
\end{lemma}
\begin{proof}
  Let $(C/T,(\ce_i,b_i,\varphi_i \colon \cf \to \ce_i^2)_{i=1}^m)$ be a multiple root. After forgetting the roots in the complement of $J$, use the partial isomorphisms between the remaining roots induced by $\varphi_i$'s to glue together a new sheaf $\cf'$ and new maps $\varphi_i' \colon \cf' \to \ce_i^2$ for $i \in J$.
\end{proof}
\begin{remark}\label{rem:forget}
  It is clear from the proof that there is a natural map $\tau\colon\cf \to \cf'$ commuting with $\varphi_i$'s and $\varphi_i'$'s.
\end{remark}

\begin{theorem}\label{thm:algebraic}
  The moduli space $\csnmbar$ is an algebraic stack, locally of finite type over $\cm$.
\end{theorem}
\begin{proof}
  The proof is by induction on $m$, with the case $m=1$ taken care of in Section~\ref{sec:single_roots}. Let us write $\cy = \csnmmbar{m-1}\times_{\cm} \csnmmbar{}$ which is an algebraic stack, locally of finite type over $\cm$ by induction hypothesis. Use the forgetful maps of Lemma~\ref{lem:forget} corresponding to the subsets $J =\{1,\dots,m-1\}$ and $J=\{m\}$ respectively to obtain a map $\csnmbar \to \cy$.

  Let $\ca' \to \cy$ be the category consisting of tuples $(C/T,\cf)$ where $\cf$ is a quasi-coherent sheaf on $C$, finitely presented and $T$-flat with $T$-proper support. Using~\cite{hall-openness-of-versality} we conclude that $\ca'$ is an algebraic stack, locally of finite type over $\cy$.

  A morphism $T \to \ca'$ corresponds to a tuple $(C/T, (\ce_i,b_i,\varphi_i \colon \cg \to \ce_i^2)_{i=1}^{m-1},(\ce_m,b_m),\cf)$. Consider the category $\ca$ over $\ca'$ of tuples of the form $(T\to \ca',\tau_1,\tau_2)$ where $\tau_1 \colon \cf \to \cg$ and $\tau_2\colon \cf \to \ce_m^2$ are morphisms such that $b_i\circ \varphi_i \circ \tau_1 = b_m \circ \tau_2$ for each $i=1,\dots,m-1$. Since all relevant modules are finitely presented, $\ca$ is an algebraic stack. Let us define $\varphi_i' \colonequals  \varphi_i \circ \tau_1$ for $i = 1,\dots,m-1$ and $\varphi_m' \colonequals  \tau_2$. We can define the subcategory $\ca_1$ of $\ca$ where $(\varphi_i')_{i=1}^m$ satisfies the pre-sync condition (Definition~\ref{def:pre-sync}) and the subcategory $\ca_2$ of $\ca_1$ where $(\varphi_i')_{i=1}^m$ satisfies---in addition---the sync condition (Definition~\ref{def:sync}).

  In light of Remark~\ref{rem:forget} the forgetful map $\csnmbar \to \cy$ factors through a map $\csnmbar \to \ca$ and identifies $\csnmbar$ with $\ca_2$. To conclude the proof, we will show that $\ca_1 \toi \ca$ is an open immersion and $\ca_2 \toi \ca_1$ is a closed immersion. The first of these statements is proven in Proposition~\ref{prop:open_immersion}. The second of these statements is a direct application of~\cite[Corollaire 7.7.8]{EGAIII-2}.
\end{proof}

\subsubsection{Pre-sync condition is open} \label{sec:pre-sync_is_open}

We use the setting in the proof of Theorem~\ref{thm:algebraic}. Let us recall that a morphism $T \to \ca$ corresponds to a tuple $(C/T, (\ce_i,b_i,\varphi_i \colon \cg \to \ce_i^2)_{i=1}^{m-1},(\ce_m,b_m),\cf, (\tau_1,\tau_2))$. We will prove that the subcategory $\ca_1 \subset \ca$ where $(\tau_1,\tau_2)$ induces a pre-sync condition is open in the algebraic stack $\ca$.

\begin{notation}
  For any $T \to \ca$ we define the loci $U_i(T), V_i(T) \subset C_T$ for $i=1,2$ in the following manner: Let $U_i(T)$ be the locus where $\tau_i$ is an isomorphism. Let $V_1(T)$ be the complement of the locus where $\cf'$ is free but $\ce_m$ is not. Let $V_2(T)$ be the complement of the locus where $\ce_m$ is free but $\cf'$ is not.
\end{notation}

  Whenever we are working locally on $C$, we may assume $m=2$ since $\cf'$ is locally isomorphic to one of $\ce_i^2$. When $m=2$ then $\cf'= \ce_1^2$. With this remark in mind we will assume $m=2$.  In this case, we simply have a pair of maps $(\varphi_i \colon\cf \to \ce_i^2)_{i=1}^2$ satisfying $b_1 \circ \varphi_1 = b_2 \circ \varphi_2$ and that our goal is to show the second pre-sync condition defines an open locus on the base $T$. Note that the second pre-sync condition is equivalent to having $U_i(T) = V_i(T)$.

  \begin{lemma}\label{lem:V_is_open}
  The loci $U_i$ and $V_i$ are open and respect base change. Precisely, for any $S \to T \to \ca$ we have $V_i(S) = V_i(T)|_S$ and $U_i(S) = U_i(T)|_S$.
\end{lemma}
\begin{proof}
  The complement of $V_i$ is the locus of points for which $\ce_i$ is free but $\ce_j$ is not ($j \neq i$). This locus is supported on the discriminant locus.  We show in Lemma~\ref{lem:constant-rank} that the rank of a root is constant on each connected component of the discriminant locus. Thus $V_i^c$ is a union of components of the discriminant locus, which is closed.  Moreover, the condition of being locally free or not behaves well with respect to base change. Therefore it is clear that $V_i(S)=V_i(T)|_S$.

  The fact that the $U_i$'s respect base change is a consequence of the following general fact. Let $\psi \colon F \to E$ be a map of finitely presented modules on $C_T$. Then the set where $\psi$ is an isomorphism is the intersection $\set{\ker \psi = 0} \cap \set{\coker \psi = 0}$. When $E$ is flat over $T$ then for any $T' \to T$ we have $\set{(\ker \psi)|_{T'} = 0} \cap \set{(\coker \psi)|_{T'} = 0} = \set{\ker (\psi|_{T'}) = 0} \cap \set{\coker (\psi|_{T'}) = 0}$. The zero locus of a finitely generated module is open and respects base change. 
\end{proof}

\begin{lemma}\label{lem:constant-rank}
  Let $C\to T$ be a stable curve and $\ce$ a locally self-dual rank-1 torsion-free module on $C$. Then the rank of $\ce$ is constant on each component of the discriminant locus $Z \subset C$. 
\end{lemma}
\begin{proof}
  Pick any point $\xp\in Z$ and note $\rk \ce|_\xp$ is either $1$ or $2$. By semi-continuity, the locus where the rank of $\ce$ is $1$ is open in $C$ and therefore on $Z$. It remains to show that the locus in $Z$ where the rank of $\ce$ is $2$ is also open.

  By making an \'etale base change and passing to an \'etale neighbourhood $U$ of $\xp$, we may assume that $T=\spec R$ and $U=\spec A$ where $\exists x,y \in A$, $\pi=xy \in R$ such that $Z|_U$ is defined by the ideal $(x,y)$. With $U$ chosen appropriately, we may apply Faltings' classification~\cite{faltings-bundles} and conclude that $E=\ce|_U$ is either free or of the form $E(p,p)$ with $p^2=\pi$ (see the notation in \emph{loc.\ cit.}\ or Appendix~\ref{app:local_deformations}). The discriminant locus $Z|_U$ is isomorphic to $\spec R/(\pi)$ and $E(p,p)$ is free at a point iff $p$ is invertible there. This is impossible on $Z|_U$ since $p^2 = 0 \imod \pi$. 
\end{proof}

\begin{lemma}\label{lem:one-direction}
  If $V_1 \subset U_1$ then $U_2 \subset V_2$. Similarly with the indices swapped. 
\end{lemma}
\begin{proof}
  Assuming $V_1 \subset U_1$ we have $V_2^c \subset V_1 \subset U_1$ by definitions. Therefore, if $\exists x \in V_2^c \cap U_2$ then $x \in U_1 \cap U_2$. But this is a contradiction, if $\varphi_1$ and $\varphi_2$ are isomorphisms at $x$ then $\ce_i$'s are isomorphic at $x$. On the other hand $x \in V_2^c$ implies that the roots have different ranks at $x$.
\end{proof}

\noindent The following proposition proves that $\ca_1 \to \ca$ is an open immersion.

\begin{proposition}\label{prop:open_immersion}
  Take a map $\spec R \to \ca$. Let $\xp \in \spec R$ be a point such that $U_i(k)=V_i(k)$ for $i=1,2$ where $k$ is the residue field of $\xp$. Then there exists a Zariski open neighbourhood $W$ of $\xp$ such that $U_i(W)=V_i(W)$ for $i=1,2$. 
  \label{prop:a1-is-open}
\end{proposition}
\begin{proof}
  Let $j \colon C|_\xp \toi C$ be the inclusion of the fiber over $\xp$. The fact that $U_i(k) = U_i(R)|_{\spec k}$ implies that $U_i(R)$ is an open neighbourhood of of $j(U_i(k))$, similarly for $V_i$'s. We know $V_i(k)$'s cover the fiber $C|_\xp$ and, by hypothesis, $U_i(k)=V_i(k)$. Therefore the open sets $U_i(R)\cap V_i(R)$ for $i=1,2$ cover the fiber $C|_\xp$.

  Pick a Zariski neighbourhood $W \subset \spec R$ of $\xp$ such that the preimage of $W$ is covered by $U_i(R)\cap V_i(R)$. Shrink $W$ so that every component of the discriminant locus intersects the fiber over $\xp$. Let $Z$ be the components of the discriminant locus on which the $\ce_i$'s are both non-free. On $Z|_\xp$ the $\varphi_i$'s are isomorphisms, hence they will remain an isomorphism in a neighbourhood of $Z \cap C_{\xp} \subset C$. Shrink $W$ one last time so that the $\varphi_i$'s are isomorphisms on all of $Z$.

  We claim that $U_i(W)=V_i(W)$. By Lemma~\ref{lem:one-direction} it will be sufficient to show $V_i(W) \subset U_i(W)$ for $i=1,2$. Pick $x \in V_1(W)$ and suppose for a contradiction that $x \notin U_1(W)\cap V_1(W)$. Then $x$ must lie in $U_2(W)\cap V_2(W)$ which implies that either both $\ce_i$'s are free or both $\ce_i$ are non-free at $x$. Furthermore, $x \in U_2(W)$ implies $\varphi_2$ is an isomorphism at $x$. If both the $\ce_i$'s are free then the fact that $\varphi_i$'s commute with $b_i$'s imply that $\varphi_1$ is also an isomorphism. Hence $x \in U_1(W)$.  If the $\ce_i$'s are both non-free, then $x \in Z$. But, by our construction of $W$, $x \in Z$ implies that $\varphi_1$ is an isomorphism at $x$.
\end{proof}

\section{Local analysis}\label{sec:local_analysis}

Fix a family of stable genus $g$ curves $\cc \to \cm$, together with a line bundle $\cn$ on $\cc$. The family $\cc \to \cm$ corresponds to a map $\cm \to \mgbar$. For this section we assume that $\cm$ is a Deligne--Mumford stack. We work over an excellent base scheme $S$ defined over $\Zz[1/2]$. A natural choice is to take $\cm = \mgbar$ and $S = \spec \Zz[1/2]$. Our main result in this section is the following theorem, whose proof is at the end of Section~\ref{sec:patch_deformations}.

\begin{theorem}\label{thm:smooth}
  If the moduli map $\cm \to \mgbar$ is smooth, then $\csnmbar$ is smooth over the base scheme~$S$.
\end{theorem}

In order to facilitate the study of the coarse moduli space associated to $\csnmbar$, we will study the action of the automorphism groups of objects in $\csnmbar$ on their local deformation functors. We end this section with a purely combinatorial description of this action.

\subsection{Patching local deformations} \label{sec:patch_deformations}

For general notions regarding deformation theory using Artin rings we refer to Schlessinger's original work~\cite{schlessinger} and Sernesi's textbook~\cite{sernesi-deformation}. See Appendix~\ref{app:local_deformations} for the deformation theory of nodes and roots near nodes.

Smoothness of $\csnmbar \to S$ can be checked around geometric points of $S$. Fix an algebraically closed field $k$ and a $k$-valued point $\xi$ of $\csnmbar$. Let $p \in \cm(k)$ and $q \in \mgbar(k)$ be the images of $\xi$. Then $q$ corresponds to a stable curve $X/k$ and $p$ induces a line bundle $\cn_X$ on $X$. Let $x_1,\dots,x_n \in X$ be the nodes of $X$ and suppose the nodes are indexed so that at least one of the roots of $\cn_X$ appearing in $\xi$ is non-free at $x_i$ precisely when $i \le n'$ for some $n' \le n$. 

Let $s \in S(k)$ be the image of $\xi$ and write $\Lambda$ for the complete local ring $\hat \co_{S,s}$ (Definition~\ref{def:local_rings}). Denote by $\art_\Lambda$ the category of local Artinian $\Lambda$-algebras with residue field $k$. Let us write $D_\xi \colon \art_\Lambda \to \Sets$ for the functor of infinitesimal deformations of $\xi$, this functor is pro-represented by the complete local ring of $\csnmbar$ at $\xi$. We will compute this pro-representing ring by breaking the deformation functor $D_\xi$ into simpler pieces.

Denote $\xi$ by the tuple $(X/k, \cn_X, \bar \cR)$ where $\bar\cR=(\bar \ce_i,\bar b_i,\bar \varphi_i)_{i=1}^m$ is a multiple root of $\cn_X$ on $X$. A deformation of $\xi$ over $R \in \art_{\Lambda}$ begins with a deformation $(\cx/R,\cn_\cx,\iota)$ of $(X/k,\cn_X)$ where $\iota$ is an identification of the central fiber of $\cx \to \spec R$ with $X \to \spec k$. In addition, we must have a deformation of the root $\bar \cR$ which will consist of a multiple root $\cR \colonequals  (\ce_i,b_i,\varphi_i)_{i=1}^m$ of $\cn_\cx$ and a sequence of isomorphisms $(j_i \colon \ce_i|_X \to \bar \ce_i)_{i=1}^m$ compatible with the sync data $(\varphi_i)_{i=1}^m$ and $(\bar\varphi_i)_{i=1}^m$.

Let $D_p \colon \art_\Lambda \to \Sets$ be the functor of infinitesimal deformations of $p$ over $\cm$. Forgetting the deformation of the multiple root $\bar \cR$ we obtain a map $D_\xi \to D_p$ which is the local version of the forgetful map $\csnmbar \to \cm$.

Let $x_v \in X$ be a node of $X$ and $\hat x_v\colon \spec \hat \co_{X,v} \to X$ be the complete local neighbourhood of $x_v$ in $X$. By pulling back the multiple root $\bar \cR$ via $\hat x_v$ we obtain a local multiple root $\bar\cR_v \colonequals  \hat x_v^* \bar\cR$ on $\hat \co_{X,v}$.

\begin{notation}
  Let $G_v \colon \art_\Lambda \to \Sets$ denote the functor of infinitesimal deformations of the node $\hat \co_{X,x_v}$ as in Definition~\ref{def:G}. Let $H_v \colon \art_\Lambda \to \Sets$ denote the functor of infinitesimal deformations of the node $\hat \co_{X,x_v}$ together with the local multiple root $\bar \cR_v$ as in Definition~\ref{def:H}.
\end{notation}
\begin{remark}
  Recall that if a root is free then it deforms trivially so that the forgetful map $H_v \to G_v$ is an isomorphism if all roots are free at $x_v$, that is, if $v > n'$. See Lemma~\ref{lem:free_roots_deform_trivially}.
\end{remark}

Given a deformation $(\cx/R, \cn_\cx, \iota , (\ce_i,b_i,\varphi_i,j_i)_{i=1}^m)$ of $\xi$, we can first forget the roots and then pass to a local neighbourhood of $x_v$ or we can first pass to a local neighbourhood of $x_v$ in $\cx$ and then forget the roots. As a result, we obtain the map
\begin{equation}\label{eq:pieces}
  D_\xi \to  D_{p} \times_{G_1} H_1 \times_{G_2} H_2 \cdots \times_{G_{n'}} H_{n'}.
\end{equation}

\begin{lemma}\label{lem:fpqc_descent}
  The natural transformation (\ref{eq:pieces}) is an isomorphism. In particular, a deformation of a multiple root is completely recovered by its deformations around the nodes.
\end{lemma}

\begin{proof}
  Fix an element $(\cx/R, \cn_\cx, \iota)$ of $D_{p}(R)$. For each $1 \le v \le n'$ pick a deformation $\cR_v \colonequals  (E_{v,i},b_{v,i},\varphi_{v,i},j_{v,i})_{i=1}^m$ on $\hat \co_{\cx,x_v}$ of the local multiple root $\hat \cR_v$. We will to show that there exists a unique deformation on $\cx$ of the multiple root $\bar \cR$ which pulls back to $\cR_v$ on the complete local neighbourhood of $x_v$ in $\cx$.

  Let $\xm_R$ be the maximal ideal of $R$. Let $R_l \colonequals  R/\xm_R^l$ and $\cx_l \colonequals  \cx|_{R_l}$ for all $l \ge 0$.  For each $l$ and $v \le n'$ we can pullback $\cR_v$ to the formal neighbourhood of $x_v$ in $\cx_l$. We will denote this deformation of $\bar \cR_v$ by $\cR_{v,l}$. 

  Using induction, we fix $N \ge 0$ and suppose that there is a unique deformation of the multiple root $\bar\cR$ on $\cx_N$ such that for all $v \le n'$ this deformation agrees with $\cR_{v,N}$ around the node $x_v$.

  By constructing a lift of this deformation to $\cx_{n+1}$ and showing that this lift is unique up to a unique isomorphism will end the proof. We will do this by fpqc-descent on $\cx_{n+1}$. The synchronized roots around the formal neighbourhoods of the nodes are one portion of the descent data.  For the rest of the descent data, we will construct the root away from the nodes and then show compatibility.

  On the complement $W$ of the nodes $x_1,\dots,x_{n'}$, the roots we have are all free. Use Lemma~\ref{lem:free_roots_deform_trivially} and Remark~\ref{rem:free_roots_deform_trivially_always} to conclude that each root deforms uniquely in $W$. This uniqueness also proves compatibility with the formal neighbourhoods around the nodes. 
\end{proof}

It is well known~\cite{deligne-mumford} that $\hat \co_{\mgbar,q} \simeq \Lambda[[t_1,\dots,t_{3g-3}]]$ where the generators $t_i$ to correspond to the deformation of the node $x_i$ for $i = 1, \dots,n$, the labeling of the rest of the generators correspond to deformations of the components of the normalization of $X$. 

With $n' \le n$ defined as above, let us define a finite extension $\Lambda[[t_1,\dots,t_{3g-3}]] \to \Lambda[[\tau_1,\dots,\tau_{3g-3}]]$ so that $t_i \mapsto \tau_i^2$ when $i \le n'$ and $t_i \mapsto \tau_i$ when $i > n'$.

\begin{proposition}\label{prop:local_csnmbar}
  The complete local ring of $\csnmbar$ at $\xi$ is isomorphic over $\Lambda$ to the tensor product $\hat\co_{\cm,p}\otimes_{\Lambda[[t_1,\dots,t_{3g-3}]]} \Lambda[[\tau_1,\dots,\tau_{3g-3}]]$.
\end{proposition}

\begin{proof}
  We computed in Appendix~\ref{app:local_deformations} that $G_v$ is pro-represented by $\Lambda[[t_v]]$ and $H_v$ is pro-represented by $\Lambda[[\tau_v]]$ with the map $H_v \to G_v$ given by $t_v \mapsto \tau_v^2$ if $v \le n'$ and $t_v \mapsto \tau_v$ if $v > n'$. The isomorphism of the map (\ref{eq:pieces}) proved in Lemma~\ref{lem:fpqc_descent} finishes the proof.
\end{proof}

Now we are ready to prove the main theorem of this section. We will stick to the notation of Proposition~\ref{prop:local_csnmbar}.

\begin{proof}[Proof of Theorem~\ref{thm:smooth}]
  If $\cm \to \mgbar$ is smooth then $\hat\co_{\cm,p}$ is smooth over $\Lambda[[t_1,\dots,t_{3g-3}]]$. Since smoothness is stable under pullbacks, the tensor product representing the local ring of $\csnmbar$ at $\xi$ is smooth over $\Lambda[[\tau_1,\dots,\tau_{3g-3}]]$, which in turn is smooth over $\Lambda$. Being smooth is stable under composition and it follows that $\csnmbar \to S$ is smooth.
\end{proof}

\subsection{Automorphism groups}\label{sec:aut_groups}

The complete local ring of $\mgbar$ at $X$ is $\Lambda[[t_1,\dots,t_{3g-3}]]$. The moduli map $\cm \to \mgbar$ gives rise to the map $\Lambda[[t_1,\dots,t_{3g-3}]] \to \hat \co_{\cm,p}$. For each $i=1,\dots,n$ let us denote the image of $t_i$ in $\hat \co_{\cm,p}$ by $\bar t_i$. As a result, we can reinterpret Proposition~\ref{prop:local_csnmbar} as follows:
\begin{equation}\label{eq:local}
  \hat \co_{\csnmbar,\xi} \simeq \hat \co_{\cm,p}[\tau_1,\dots,\tau_{n'}]/(\tau_1^2 - \bar t_1, \dots ,\tau_{n'}^2 - \bar t_{n'}).
\end{equation}
Scaling $\tau_i$ by $\pm 1$ gives an automorphisms of the ring $\hco_{\csnmbar,\xi}$ fixing $\hco_{\cm,p}$. 

There is an action of $\aut(\xi)$ on $\hco_{\csnmbar,\xi}$ defined as follows. The group $\aut(\xi)$ acts on the functor of infinitesimal deformations $D_\xi$ of $\xi$ by changing the identification of the central fiber of a deformation through post-composition via an automorphism of the central fiber. Since $\hco_{\csnmbar,\xi}$ pro-represents $D_\xi$, the group $\aut(\xi)$ acts on it. Any subgroup of $\aut(\xi)$ fixing $(X,\cn_X)$ over $\cm$ will fix $\hco_{\cm,p}$.

\begin{definition}\label{def:aut0}
  The subgroup of $\aut(\xi)$ fixing $(X,\cn_X)$ in $\cm$ will be denoted by $\aut_0(\xi)$ and will be called the group of \emph{inessential automorphisms}, in accordance with~\cite{cornalba, caporaso-casagrande-cornalba}. By the paragraph above, we get a natural morphism 
  \begin{equation}\label{eq:inessential_action}
    \rho\colon \aut_0(\xi) \to \aut(\hco_{\csnmbar,\xi}/\hco_{\cm,p}).
  \end{equation}
\end{definition}

Let $W=\{x_1,\dots,x_{n'}\} \subset X$ be the set of nodes where at least one root in $\xi$ is non-free. Denote by $X_W \to X$ the partial normalization of $X$ at $W$. We will write $V$ for the set of connected components of $X_W$. 

\begin{notation}\label{not:graph}
  Let $\Gamma$ be the graph with vertex set $V$ and edge set $W$, the ends of an edge $w \in W$ are the components of the partial resolution $X_W \to X$ containing a preimage of $w$.
\end{notation}

We will consider the cohomology of the graph $\Gamma$ with coefficients in the field $\Ff_2$ with two elements. Note that we do not need an orientation on the edges of $\Gamma$ when working with these coefficients. For our applications, it makes sense to identify $\Ff_2$ with the set $\{\pm 1\}$ using the bijection $\Ff_2 \to \{\pm 1\}; u \mapsto (-1)^u$. Through this bijection we endow the set $\{\pm 1\}$ with a field structure and denote the resulting field by $\Ff_2^{\pm}$.

Let $C^0(\Gamma) = \Ff_2^{\pm}\langle V \rangle$ and $C^1(\Gamma) = \Ff_2^{\pm}\langle W \rangle$ be the 0-chains and 1-chains of $\Gamma$ respectively. The usual coboundary map $\del \colon C^0(\Gamma) \to C^1(\Gamma)$ gives rise to the cohomology groups $\H^0(\Gamma)$ and $\H^1(\Gamma)$ which fit into the exact sequence
\begin{equation}\label{eq:graph_cohomology}
  \begin{tikzcd}
    0 \arrow[r] & \H^0(\Gamma) \arrow[r] & C^0(\Gamma) \arrow[r,"\del"] & C^1(\Gamma)  \arrow[r] & \H^1(\Gamma) \arrow[r] & 0 .
  \end{tikzcd}
\end{equation}

\begin{remark}\label{rem:C1_in_aut}
  For each element $h \in C^1(\Gamma)$ we can define an automorphism of $\hco_{\csnmbar,\xi}$ over $\hco_{\cm,p}$ by scaling $\tau_i$ by the value of $h(x_i) \in \{\pm 1\}$, this gives an inclusion
  \begin{equation}\label{eq:C1_in_aut}
   \sigma\colon C^1(\Gamma) \toi \aut(\hco_{\csnmbar,\xi}/\hco_{\cm,p}).
  \end{equation}
\end{remark}

Define a map $\psi \colon \aut_0(\xi) \to C^1(\Gamma)$ by the following rule. For each $a \in \aut_0(\xi)$ we need to define a value $a(x) \in \{\pm 1\}$ for each $x \in W$. Given $x \in W$, there is at least one index $i \le m$ such that the $i$-th root $(\bar \ce_i,b_i)$ is not locally free at $x$. Pull back $(\bar \ce_i,\bar b_i)$ to the formal neighbourhood $\hco_{X,x}$ and apply Lemma~\ref{lem:iso-class} to conclude that the action of the automorphism $a$ on the $i$-th root around $x$ can be described by a matrix of the form $\begin{psmallmatrix} \varepsilon_1 & 0  \\  0 & \varepsilon_2  \end{psmallmatrix}$ where $\varepsilon_1,\varepsilon_2 \in \{\pm 1\}$. Let us define $\psi(a)(x) = \varepsilon_1 \varepsilon_2$.

\begin{lemma}
  The value of $\psi(a)(x)$ is well defined.
\end{lemma}
\begin{proof}
  Although the sign $\varepsilon_1\varepsilon_2$ obtained from Lemma~\ref{lem:iso-class} is not well defined for an isomorphism between two different roots, it is well defined for an automorphism of a root. As for the choice of $i$ used in defining $\psi(a)(x)$, the sync data forces the sign $\varepsilon_1\varepsilon_2$ to be independent of this choice, see Appendix~\ref{app:multiple_roots} for more details on local multiple roots.
\end{proof}

\begin{theorem}\label{thm:aut_action}
  The map $\psi$ defined above commutes with the natural maps defined in (\ref{eq:C1_in_aut}) and (\ref{eq:inessential_action}). That is, the following diagram is commutative:
  \begin{equation}
    \begin{tikzcd}
      \aut_0(\xi) \arrow[rr,"\psi"] \arrow[dr,"\rho"'] &             & C^1(\Gamma) \arrow[ld,"\sigma",hook]\\
      & \aut(\hco_{\csnmbar,\xi}/\hco_{\cm,p}) &
    \end{tikzcd}.
  \end{equation}
\end{theorem}
\begin{proof}
  Fix $a \in \aut_0(\xi)$ and $x_i \in W$. We will denote the parity by $c\colonequals  \varepsilon_1\varepsilon_2 = \psi(a)$. Consider the pullback of $\xi$ to $\hco_{X,x_i}$ and its universal deformation as in Appendix~\ref{app:multiple_roots}. We show in Theorem~\ref{thm:universal-def-of-mult-roots} that this universal deformation has base $\Lambda[[\tau_i]]$. The automorphism $a$ induces an action on this base ring $\Lambda[[\tau_i]]$ defined as in the description leading up to Definition~\ref{def:aut0}. It suffices to show that this induced action consists of scaling $\tau_i$ by $c$.

  Since $\tau_i \neq 0$ this change in sign is required to accommodate an isomorphism between the two universal families of local multiple roots whose square has parity $c$, as in Definition~\ref{def:parity}. This follows directly from the arguments given in Appendix~\ref{app:multiple_roots}.
\end{proof}

Now we want to express the image of the map $\psi\colon \aut_0(\xi) \to C^1(\Gamma)$ in a more combinatorial fashion. For each $i=1,\dots,m$ we can forget all but the $i$-th root in $\xi$ in order to obtain $\xi_i \colonequals  (X,\cn_X,\bar \ce_i,\bar b_i)$. We will write $W_i \subset W$ for the set of nodes on which $\bar \ce_i$ is not free and by $V_i$ the irreducible components of the normalization $X_{W_i} \to X$ of $X$ at $W_i$. As in Notation~\ref{not:graph} we define a graph $\Gamma_i = (V_i,W_i)$. Moreover, the construction of the map $\psi$ applies also when $m=1$ and in particular to $\xi_i$ giving us maps $\psi_i \colon \aut_0(\xi_i) \to C^1(\Gamma_i)$. We have exact sequences as in (\ref{eq:graph_cohomology}) obtained by replacing $\Gamma$ with $\Gamma_i$.

In $\aut_0(\xi)$ and $\aut_0(\xi_i)$ we have a distinguished automorphism which scales all roots by $-1$. Let $\{\pm 1\}$ denote the corresponding subgroup generated by this element. It is clear that both $\psi$ and $\psi_i$ contain the subgroup $\{ \pm 1\}$ in their kernels, which prompts the following notation.

\begin{notation}
  Let $\autu_0(\xi)= \aut_0(\xi)/\{\pm 1\}$ and $\autu_0(\xi_i) = \aut_0(\xi_i)/\{\pm 1\}$ for $i=1,\dots,m$.
\end{notation}

\begin{lemma}[\cite{cornalba, caporaso-casagrande-cornalba}]\label{lem:aut-i}
  For $i=1,\dots,m$ we have $\aut_0(\xi_i) \simeq C^0(\Gamma_i)$ and $\psi_i$ establishes an isomorphism $\autu_0(\xi_i) \isoto \ker(C^1(\Gamma_i) \to \H^1(\Gamma_i))$.
\end{lemma}

\begin{proof}
  Let $\nu_i\colon X_{W_i} \to X$ be the partial normalization of $W_i \subset X$. Let $P_i = \nu_i^*(W_i)$ and $N_i \colonequals  \nu_i^*(\cn_X)(-P_i)$. Jarvis shows in \S 4.1.1~\cite{jarvis-torsion-free} that there exists a line bundle $L_i$ on $X_{W_i}$, and a squaring map $\beta_i: L_i^{\otimes 2} \isoto N_i$ such that $\bar \ce_i \simeq \nu_{i,*}L_i$ and $\bar b_i$ is obtained from $\beta_i$. Furthermore, he shows that $\aut(\bar \ce_i,\bar b_i) \simeq \aut(L_i,\beta_i)$.

  The automorphism group of $(L_i,\beta_i)$ is clearly isomorphic to $\H^0(X_{W_i},\mu_2)$ where $\mu_2 \subset \co_{X_{W_i}}^*$ is the kernel of the squaring map. We have an obvious identification $\H^0(X_{W_i},\mu_2) \simeq C^0(\Gamma_i)$ which proves $\aut_0(\xi_i) \simeq C^0(\Gamma_i)$.  
  Since $\Gamma_i$ is connected, $\H^0(\Gamma_i) \simeq \set{\pm 1}$. Therefore, $\autu_0(\xi_i) = \aut_0(\xi_i)/\set{\pm 1} \simeq C^0(\Gamma_i)/\H^0(\Gamma_i) \simeq \im \del$.
\end{proof}

Since elements in $\aut_0(\xi)$ act on individual roots, we have natural maps $\aut_0(\xi) \to \aut_0(\xi_i)$ obtained by restricting the action of an automorphism of $\xi$ to just the $i$-th root. Any automorphism of multiple roots can be recovered from its action on the individual roots so the combined map $\aut_0(\xi) \to \prod_{i=1}^m \aut_0(\xi_i)$ is an injection.

We can also define maps $C^1(\Gamma) \tos C^1(\Gamma_i)$ obtained by sending all edges in $W\setminus W_i$ to zero. The joint map $C^1(\Gamma) \to \prod_{i=1}^m C^1(\Gamma_i)$ is an injection since each edge in $W$ appears in at least one $W_i$.

\begin{lemma}\label{lem:cartesian_auts}
  With the natural maps described above, we obtain the following \emph{Cartesian} diagram:
  \[
    \begin{tikzcd}
      \autu_0(\xi_0) \arrow[r,hook]\arrow[d,"\psi_0"']\arrow[dr, phantom, "\ulcorner", very near start] & \prod_{i=1}^m \autu_0(\xi_i) \arrow[d,"\prod \psi_i",hook] \\ 
      C^1(\Gamma) \arrow[r,hook] & \prod_{i=1}^m C^1(\Gamma_i)
    \end{tikzcd}
  \]
\end{lemma}
\begin{proof}
  The commutativity is immediate since the value of $\psi(a)(x)$ is determined by the action of $a$ on any any root $(\bar \ce_i, \bar b_i)$ which is not free on $x$. To see that the diagram is Cartesian, we observe that a sequence of automorphisms $(a_i)_{i=1}^m$ where $a_i$ acts on $\xi_i$ are compatible with the sync data if the image of $(a_i)_{i=1}^m$ lies in the image of $C^1(\Gamma)$.
\end{proof}

\begin{proposition}\label{prop:combinatorial_aut}
  The map $\psi$ identifies $\autu_0(\xi)$ with the kernel of the map $C^1(\Gamma) \to \prod_{i=1}^{m} \H^1(\Gamma_i)$.
\end{proposition}
\begin{proof}
  By Lemma~\ref{lem:cartesian_auts} we conclude that $\autu_0(\xi)$ is the intersection of $C^1(\Gamma)$ with $\prod_{i=1}^m \autu_0(\xi_i)$ in $\prod_{i=1}^m C^1(\Gamma_i)$. Now apply Lemma~\ref{lem:aut-i}.
\end{proof}

\begin{example}
  Let us consider Example~\ref{ex:red}. Here $\Gamma=\Gamma_1=\Gamma_2$ consists of two vertices and one edge between them. Therefore, $C^1(\Gamma)\simeq \Ff_2$ and $\prod_{i=1}^2 \H^1(\Gamma_i) = 0$. We conclude that $\autu_0(\xi)\simeq \Ff_2$. In \eqref{eq:local} we have $n'=1$ so $\hat \co_{\csnmbar,\xi} \simeq \hat \co_{\cm,p}[\tau_1]/(\tau_1^2 - \bar t_1)$ and the generator of $\autu_0(\xi)$ acts by $\tau_1\mapsto - \tau_1$. In particular, the coarse moduli space of double spin curves is not branched around $\xi$ over the moduli space of curves.
\end{example}

\begin{example}
  Consider Example~\ref{ex:irred}. Here $\Gamma=\Gamma_1=\Gamma_2$ consists of one vertex and a loop. Therefore, $C^1(\Gamma) \to \prod_{i=1}^2 \H^1(\Gamma_i)$ becomes $\Ff_2 \to \Ff_2^2 : 1 \mapsto (1,1)$. In particular, $\autu_0(\xi)=0$. Using \eqref{eq:local} we see that the locus of such $\xi$ form a divisorial branch locus of the coarse moduli of double spin curves over the moduli space of curves.
\end{example}

\section{Fundamental properties of the moduli of multiple-roots}\label{sec:fundamental_properties}

In this section we prove the additional structural results regarding $\csnmbar\to \cm$ in order to complete the proof of Theorem~\ref{thm:intro_compactify}. Throughout this section we will assume $\cm$ is a Deligne--Mumford stack.

\subsection{Proper, Deligne--Mumford compactification}

  As we proved in Theorem~\ref{thm:algebraic} that $\csnmbar$ is algebraic, it follows that the relative diagonal $\csnmbar \to \cm$ is representable by algebraic spaces.

\begin{lemma}\label{prop:diagonal_finite_unramified}
  The diagonal $\Delta \colon \csnmbar \to \csnmbar \times_\cm \csnmbar$ is finite and unramified.
\end{lemma}

\begin{proof}
  Fix a morphism $B \to \csnmbar \times_\cm \csnmbar$ where $B$ is a scheme. This defines a curve $C \to B$ and a pair of multiple roots $\cR = (\ce_i,b_i,\varphi_i)_{i=1}^m$ and $\cR' = (\ce_i',b_i',\varphi_i')_{i=1}^m$ on the curve $C$. We claim that the isomorphism functor $\isom_C(\cR,\cR')$ is represented by a finite and unramified scheme over~$B$. 
  
  An isomorphism of a multiple root is an isomorphism of the underlying sequence of roots compatible with the sync data. Let $\cR_i=(\ce_i,b_i)$ and $\cR_i'=(\ce_i',b_i')$ be the $i$-th roots and consider the natural injection $\isom_C(\cR,\cR') \to \prod_{i=1}^m \isom_C(\cR_i,\cR_i')$, where the product of the functors $\isom_C(\cR_i,\cR_i')$ is to be taken over $B$. This product is finite and unramified since each component $\isom_C(\cR_i,\cR_i')$ is represented by a finite and unramified scheme over $B$ as shown in \S 4.1.4.3 of~\cite{jarvis-torsion-free}. We claim $\isom_C(\cR,\cR')$ is a component of the product and we will perform a series of reductions to prove this.

  Since the stacks under consideration are locally noetherian and the question is local on the target of $\Delta$ we may assume $B$ to be noetherian. Since the diagonal is representable, $\isom_C(\cR,\cR')$ is locally noetherian and we may reduce to the case where $B=\spec R$ where $R$ is a complete discrete valuation ring.  Let the $\eta$ and $\sigma$ denote the generic and special points of $B$ respectively. To conclude the proof we need only prove that given a sequence of isomorphisms between the roots, the property of being compatible with the sync data both specializes and generalizes. This problem requires studying pairs of roots at a time and therefore we will assume $m=2$.

  Consider a map $h\colon \spec R \to \prod_{i=1}^m \isom_C(\cR_i,\cR_i')$, which gives us a family of curves $\cx \to \spec R$ and a sequence of isomorphism $(h_i \colon (\ce_i',b_i') \to (\ce_i',b_i'))_{i=1}^m$  between the roots. We need to check that the sequence $(h_i)_{i=1}^m$ is compatible with the sync data over the generic fiber iff it is compatible with the sync data over the special fiber.

  Compatibility is trivially satisfied except around the nodes where both roots may be singular. In this respect, Lemma~\ref{lem:fpqc_descent} implies that compatibility need only be checked around the formal neighbourhood of the node of the special fiber. In Lemma~\ref{lem:iso-class} we show that in the formal neighbourhood of a node, the isomorphisms $h_i$ between $\ce_i$ and $\ce_i'$ are of the form $\begin{bsmallmatrix}  \varepsilon_i'& 0  \\  0 & \varepsilon_i'' \end{bsmallmatrix}$, with $\varepsilon_i',\varepsilon_i'' \in \set{\pm 1}$. The symmetric square $h_i^2$ equals $\begin{bsmallmatrix} 1 & 0 & 0 \\  0 & \varepsilon_i & 0 \\ 0 & 0 & 1  \end{bsmallmatrix}$, where $\varepsilon_i = \varepsilon_i' \varepsilon_i''$. 

  In this formal neighbourhood of the node, we can replace the sync data with an isomorphism $\psi: \ce_1^2 \to \ce_2^2$ on the first pair of roots and $\psi'$ on the second pair of roots (see Remark~\ref{rem:replace_with_squares}). From Proposition~\ref{prop:iso_square} we know that $\psi = \begin{bsmallmatrix} 1 & 0 & 0 \\  0 & \varepsilon & 0 \\ 0 & 0 & 1  \end{bsmallmatrix}$ and $\psi' = \begin{bsmallmatrix} 1 & 0 & 0 \\  0 & \varepsilon' & 0 \\ 0 & 0 & 1  \end{bsmallmatrix}$ where $\varepsilon,\varepsilon' \in \set{\pm 1}$.

  These four isomorphisms give a commuting diagram iff $\varepsilon \varepsilon_1 = \varepsilon'\varepsilon_2$. This equality holds over the special fiber iff it holds over the generic fiber. 
\end{proof}

\begin{proposition}\label{prop:dm_stack}
  The moduli space $\csnmbar$ is a Deligne--Mumford stack.
\end{proposition}
\begin{proof}
  The proposition above combined with the fact that $\cm$ is Deligne--Mumford stack yields this result.
\end{proof}

\begin{lemma} \label{lem:proper}
  The morphism $\csnmbar \to \cm$ is proper.
\end{lemma}
\begin{proof}
  We prove this by the valuative criterion of properness and induction on $m\ge 1$. The result for $m=1$ is part of Theorem~\ref{thm:m_is_1}. Then we need only show that the map $\cy\colonequals \csnmbar \to \csnmmbar{m-1}\times_{\cm}\csnmmbar{}$ constructed in the first paragraph of the proof of Theorem~\ref{thm:algebraic} is proper. As the diagonal is locally noetherian we restrict to checking the valuative criterion using complete DVRs.

  Let $R$ be a complete DVR, with residue field $K$. Consider a 2-commutative diagram:
  \[
    \begin{tikzcd}
      \spec K \arrow[r] \arrow[d] & \csnmbar \arrow[d,""{name=U,below}]\\
      \spec R \arrow[r,""{name=D,near end}]           & \cy \arrow[Rightarrow, bend right=50,from=U,to=D]
    \end{tikzcd}
  \]

  This means that we have a synchronized $(m-1)$-tuple of roots and an $m$-th root over the curve $C_R \to \spec R$. Furthermore, these roots are all synchronized over the general fiber. However, as we demonstrated in the proof of Proposition~\ref{prop:diagonal_finite_unramified} a synchronization on the generic fiber over a complete DVR extends to the entire family uniquely.
\end{proof}

\noindent We now prove that $\csnmbar$ is indeed a `closure' of $\csnm$ in the sense we would expect.

\begin{lemma} \label{lem:dense-open-immersion}
  If $\cc \to \cm$ is generically smooth then $\csnm \to \csnmbar$ is a dense open immersion.
\end{lemma}

\begin{proof}
  Since all roots on a smooth curve are locally free, it suffices to show that any root can be deformed onto a smooth curve. Provided that any singular curve $X$ can be deformed to a smooth curve over $\cm$, it follows immediately from the local deformation functors discussed in Section~\ref{sec:patch_deformations} that any tuple of roots on $X$ can also be deformed onto a smooth curve over $\cm$. In particular, this is an immediate consequence of Lemma~\ref{lem:fpqc_descent}.
\end{proof}

\begin{remark}
  If $\cc \to \cm$ is not assumed to be generically smooth the result will certainly not hold, even when $m=1$. For example one could take $\cm = \spec k$. In this case, the isomorphism classes of roots of a fixed line bundle form a discrete set. If $\cc$ is singular, then some of these roots will not be free.
\end{remark}

\subsection{The coarse moduli space}

\begin{lemma} \label{lem:quasi-finite}
 The map $\csnmbar \to \cm$ is quasi-finite.
\end{lemma}
\begin{proof}
  We will build on the fact that $\csnmmbar{} \to \cm$ is quasi-finite which is a part of Theorem~\ref{thm:m_is_1}. Fix a geometric point of the $m$-fold product $\csnmbar{} \times_{\cm} \dots\times_{\cm} \csnmmbar{}$. Our goal is to show that there are finitely many synchronizations on the corresponding sequence of roots. Using Lemma~\ref{lem:fpqc_descent} we see that we need to synchronize the roots only around the nodes. Now apply Remark~\ref{rem:replace_with_squares} to express synchronizations as a sequence of isomorphisms of the local multiple roots. Now apply Lemma~\ref{lem:iso-class} to see that there are only finitely many isomorphisms.
\end{proof}

\noindent  The proof below is adapted from Proposition 3.1.1~\cite{jarvis-geometry}.

\begin{proposition} \label{prop:projective}
  If the coarse moduli space of $\cm$ is projective over $S$ then the coarse moduli space of $\csnmbar$ is projective over $S$.
\end{proposition}
\begin{proof}
  It is well known that separated Deligne--Mumford stacks are coarsely represented by algebraic spaces (e.g. Corollary 1.3.1~\cite{keel-mori-quotient}). So we let $X = \coarse(\csnmbar)$ and $Y = \coarse(\cm)$ be these coarse moduli spaces with $f: X \to Y$ the natural map between them. 

  This map $f$ is proper because the corresponding map between the stacks is proper. Also $f$ is quasi-finite by Lemma~\ref{lem:quasi-finite}. Therefore $f$ is finite and hence projective. When $Y \to S$ is projective then so is $X \overset{f}{\to} Y \to S$.
\end{proof}

\subsection{Proof of Theorem~\ref{thm:intro_compactify}}\label{sec:proof_main}
We will first prove a general version of Theorem~\ref{thm:intro_compactify} below. To that end, let us recall all assumptions. Take $S \to \spec \Zz[1/2]$ to be an excellent scheme. Suppose that $\cm$ is a proper Deligne--Mumford stack over $S$ and $\cc \to \cm$ is a stable curve of genus $g\ge 2$. Further assume that the generic fiber of $\cc \to \cm$ over each irreducible component of $\cm$ is smooth. Take a line bundle $\cn$ on $\cc$ of absolutely bounded degree (Definition~\ref{def:bounded-degree}). 

\begin{theorem}\label{thm:main}
  With the conditions above the moduli space $\csnmbar$ is a smooth and proper Deligne--Mumford stack over the base scheme $S$. Furthermore, the inclusion $\csnm \toi \csnmbar$ is dense and open while the forgetful map $\csnmbar \to \cm$ is quasi-finite.
\end{theorem}
\begin{proof}
  The moduli space $\csnmbar$ is smooth by Theorem~\ref{thm:smooth} and Deligne--Mumford by Proposition~\ref{prop:dm_stack}. The inclusion $\csnm \toi \csnmbar$ is a dense open immersion by Lemma~\ref{lem:dense-open-immersion}. The map $\csnmbar \to \cm$ is proper by Lemma~\ref{lem:proper} and therefore $\csnmbar \to S$ is proper since $\cm \to S$ is assumed to be proper. Lemma~\ref{lem:quasi-finite} establishes the claim on quasi-finiteness.
\end{proof}

  Take $S=\spec \Zz[1/2]$ or one of its geometric points. Let $\cm = \mgnbar$ be the moduli space of marked stable curves and take the universal curve $\cgnbar \to \mgnbar$ over it with sections $\sigma_1,\dots,\sigma_n \colon \mgnbar \to \cgnbar$. Fix a tuple $(a_1,\dots,a_n)\in\Zz^n$ and define the line bundle $\cn$ to be the twisted relative dualizing sheaf $\omega_{\cgbar/\mgbar}(\sum_{i=1}^n a_i \sigma_i)$. With these conditions the moduli space $\csnmbar$ satisfies the hypotheses of Theorem~\ref{thm:main}. 

\begin{definition}\label{def:sgmbar}
  The moduli space of multiple-spin curves $\sgmbar$ is the space $\csnmbar$ where we are using the setup above with $n=0$ and $\cn=\omega_{\cc/\mg}$.
\end{definition}

Theorem~\ref{thm:intro_compactify} is a special case of Theorem~\ref{thm:main}.

\subsection{Normalization of the product}

Take $\cm = \mgbar$ with the universal curve $\cgbar \to \mgbar$ and take as the line bundle $\cn$ the relative dualizing sheaf $\omega_{\cgbar/\mgbar}$. The resulting moduli space $\csnmbar$ comes with a map $\csnmbar \to \sgtimesbar$ which forgets the synchronizations.

\begin{proposition}\label{prop:normalize}
  The moduli space $\sgtimesbar$ is not normal for $m \ge 2$ and the forgetful map $\csnmbar \to \sgtimesbar$ is the normalization.
\end{proposition}
\begin{proof}
  Cornalba \cite{cornalba} proves that there is a unique divisorial component of the ramification locus of the map $\sgbar \to \mgbar$ on each of the two components of $\sgbar$. The two components are defined by the parity of the spin structure which plays no role for the nature of our problem. 
  
  The general member of the ramification locus is a limit spin structure on an irreducible curve with a single node, with the spin structure acquiring a singularity at the node. Therefore, $\sgtimesbar$ has divisorial singularities precisely along the closure of the locus of curves whose generic member is an irreducible curve with a single node and at least two of the $m$ roots are singular at the node. For the rest of the proof we may assume $m=2$, the argument being similar for $m \ge 2$.

  Using Example 5.4.(b) of \cite{cornalba} we can describe the map $\sgtimesbarm{2} \to \mgbar$ locally around the non-normal locus, which in local coordinates is given by the map
  \begin{equation}
    k[[t_1,\dots,t_{3g-3}]] \to k[[\tau_{11},\tau_{12},t_2,\dots,t_{3g-3}]]/(\tau_{11}^2 -t_1, \tau_{12}^2-t_1).
  \end{equation}
  The normalization of the ring on the right is given by
  \begin{align}
    k[[\tau_{11},\tau_{12},t_2,\dots,t_{3g-3}]]/(\tau_{11}^2 -t_1, \tau_{12}^2-t_1) &\to \left( k[[\tau,t_2,\dots,t_{3g-3}]]/(\tau^2-t_1) \right)^{\times 2} \\
    \begin{bsmallmatrix}
      \tau_{11} \\
      \tau_{12} 
    \end{bsmallmatrix} &\mapsto 
    \begin{bsmallmatrix}
      (\tau,\tau) \\
      (\tau,-\tau)
    \end{bsmallmatrix}.\nonumber
  \end{align}
  The two components of the normalization correspond to the choice of parity (Definition~\ref{def:parity}) of the local sync data around the node. The corresponding formal neighbourhoods in $\csnmmbar{2}$ are given exactly as in the normalized components on the left. This follows from the description of the local deformation functors (Proposition~\ref{prop:local_csnmbar}) and the fact that a generic irreducible nodal curve has no automorphisms and, in this case, there are no inessential automorphisms of the roots that act non-trivially on the base of the local deformation functors (Proposition~\ref{prop:combinatorial_aut}). When there are $m \ge 2$ roots and $n \le m$ of them are singular around the node, we would get $2^{n-1}$ components from the normalization which agrees with the number of choices of parity.
\end{proof}

\section{Components of the moduli space of multiple-spin curves}\label{sec:components}

Let $\sg/k$ be the moduli space of spin curves, as described in the introduction, over an algebraically closed field $k$ of characteristic not $2$. In this section, we will classify the irreducible components of $\sgm = \sg \times_{\mg} \dots \times_{\mg} \sg$ for any $m \ge 1$. In doing so, we will also classify the irreducible components of its compactification $\sgmbar$ (Definition~\ref{def:sgmbar}). 

\begin{lemma}
  The connected components of $\sgm$ are irreducible. The number of (connected and irreducible) components of $\sgm$ does not depend on the field $k$. In fact, for any $k$ these components are in natural bijection with the components of $\sgm/\Cc$. These statements are also true for $\sgmbar$ in place of $\sgm$.
\end{lemma}
\begin{proof}
The inclusion $\sgm \toi \sgmbar$ is dense and open by Theorem~\ref{thm:intro_compactify} so that the irreducible components of $\sgm$ and $\sgmbar$ coincide. 

Theorem~\ref{thm:main} implies that we can define $\sgmbar$ over $\Zz[1/2]$, where it is still proper and smooth. Apply Theorem~4.17.(iii) of~\cite{deligne-mumford} to conclude that for any algebraically closed field $k$ of $\chr(k) \neq 2$, the number of connected components of $\sgmbar|_k$ is the same as the number of connected components of $\sgmbar|_\Cc$. The bijection of the components are attained by taking the closure of the components over $\sgmbar|_\Cc$. Alternatively, the classification of the components given in Section~\ref{sec:deformation_invariants} is characteristic independent and makes the bijection more explicit.

Finally, since the spaces in question are smooth, their connected components and irreducible components coincide. 
\end{proof}

We have thus reduced the problem to classifying the components over $k= \Cc$. We will take our base field to be $\Cc$ for the rest of this section. The components of $\sgm$ over $\Cc$ will be classified via the monodromy action on the fibers of $\sgm \to \mg$. 

\subsection{Deformation invariants of spin structures}\label{sec:deformation_invariants}

In this subsection we will describe a natural set of deformation invariants of an $m$-tuple of spin structures on a curve. The main theorem of this section states that these deformation invariants fully classify the components of $\sgt = \sg \times_{\mg} \dots \times_{\mg} \sg$. The proofs given in this subsection are outlines with proofs completed in the rest of this section.

Let $X$ be smooth proper genus $g$ curve over $\Cc$ and let $S_X$ denote the fiber of $\sg \to \mg$ over $X$, which we will view as the isomorphism classes of spin structures $(\eta,\alpha \colon \eta^{\otimes 2} \isoto \omega_{X/\Cc})$ on $X$. Up to isomorphisms, we can drop the map $\alpha$. In other words, the set $S_X$ is in bijection with the isomorphism classes of line bundles $\eta$ such that $\eta^{\otimes 2} \isoto \omega_C$. These isomorphism classes (and, by abuse of notation, the line bundles themselves) are called \emph{theta characteristics}. 

The only deformation invariant of a spin curve $(X,\eta)$ is the \emph{parity} of the theta characteristic $\eta$~\cite{mumford-theta-chars}, which is defined to be the parity of the dimension of the global sections of $\eta$. We will denote the parity by
\begin{equation}
 P(\eta) \colonequals h^0(\eta) \imod 2 \in \Ff_2.
\end{equation}
In other words, $\sg$ consists of two components and the component to which $(X,\eta)$ belongs is determined by the parity of $\eta$. 

Given three theta characteristics $\eta_1,\eta_2,\eta_3 \in S_X$, we can build a fourth one $\eta_1\otimes \eta_2 \otimes \eta_3 \otimes \omega_X^\vee = \eta_1 \otimes \eta_2 \otimes \eta_3^\vee$. 
In general, the parity of $\eta_i$'s alone do not determine the parity of $\eta_1\otimes \eta_2 \otimes \eta_3 \otimes \omega_X^\vee$ which, therefore, introduces another deformation invariant.

\begin{definition}\label{def:syzygy}
  For an  $m$-tuple $\eta_\bullet = (\eta_1,\dots,\eta_m) \in S_X^{m}$ we define the \emph{syzygy relations} $\syzy(\eta_\bullet)$ to be the parities of all $\eta_i$ and all $\eta_j \otimes \eta_k \otimes \eta_1^\vee$:
  \begin{equation}\label{eq:relation}
  \syzy(\eta_\bullet) \colonequals \left( P(\eta_i); P(\eta_j \otimes \eta_k \otimes \eta_1^\vee) \right)_{\substack{1 \le i \le m \\ 1 < j < k \le m}} \in \Ff_2^{m + {m-1 \choose 2}}. 
 \end{equation}
\end{definition}

Once we account for the relations between the theta characteristics, the syzygy relations will give a complete set of deformation invariants. Moreover, we will see that the syzygy relations will be a minimal set of deformation invariants when $g\ge m$ and when there are no relations between the theta characteristics.

\begin{definition}\label{def:eta_relations}
  A vector $c_\bullet \in \Ff_2^m$ with $\sum c_i = 0$ is a \emph{relation} for a tuple $\eta_\bullet \in S_X^{m}$ if
  \begin{equation}
    \eta_1^{c_1} \otimes \dots \otimes \eta_m^{c_m} \simeq \omega_X^{k} \quad \text{ for some } k \in \Zz.
  \end{equation}
  The set of relations $R(\eta_\bullet) \subset \Ff_2^m$ of $\eta_\bullet$ is a vector space. If $R(\eta_\bullet) = 0$ then $\eta_\bullet$ (as well as $(X,\eta_\bullet)$) will be called \emph{non-degenerate}.
\end{definition}

\begin{remark}
  Note that the condition $\sum c_i =0$ is redundant if $g\neq 1$ by degree considerations. Also the lift of $c_i \in \Ff_2$ to $\Zz$ in taking a power of $\eta_i$ influences only the value of $k$, but is otherwise immaterial.
\end{remark}

\begin{remark}
  If $c_\bullet$ is a relation for $\eta_\bullet$ and $c_i \neq 0$ then
  \begin{equation}
    \eta_i \simeq \omega_X^{k'} \bigotimes_{j \neq i} \eta_j^{c_j}
  \end{equation}
  for some $k'$. In other words, if $R(\eta_\bullet) \neq 0$ then $\eta_\bullet$ can be reconstructed from a non-degenerate sub-tuple of $\eta_\bullet$.
\end{remark}

A portion of the main result of this section can be stated as follows.

\begin{theorem}\label{thm:def_inv}
The pair $(R(\eta_\bullet),\syzy(\eta_\bullet))$ is a complete deformation invariant of $(X,\eta_\bullet)$. That is, two elements $(X,\eta_\bullet), (X',\eta_\bullet') \in \sgt$ lie in the same component of $\sgt$ if and only if $R(\eta_\bullet) = R(\eta_\bullet')$ and $\syzy(\eta_\bullet) = \syzy(\eta_\bullet')$. 
\end{theorem}

\begin{proof}
  The deformation invariance of $(R,\syzy)$ is Proposition~\ref{prop:def_inv}. In particular, we can deform $(X',\eta_\bullet')$ to some $(X,\eta_\bullet'')$ while keeping $(R,\syzy)$ constant. Now we use the monodromy action of $\sgm \to \mg$ over $X$ to argue that $(X,\eta_\bullet)$ and $(X,\eta_\bullet'')$ lie on the same component if and only if their space of relations and their syzygies coincide. The last step combines Section~\ref{sec:teich-spin} (Corollary~\ref{cor:monodromy-to-orbits}) and Section~\ref{sec:affine-geometry} (Corollary~\ref{cor:degenerate_orbit}). 
\end{proof}

The theorem above can be strengthened by describing exactly which syzygy relations can occur. We begin by shortening the syzygy relations by using a non-degenerate sub-tuple.

\begin{definition}\label{def:reduced_syzygy}
  Given $\eta_\bullet \in S_X^{m}$, there is a unique non-degenerate sub-tuple $\widetilde \eta_\bullet = (\eta_{i_1},\dots, \eta_{i_k})$ where $k = m-\dim R(\eta_\bullet)$ and $i_1 < \dots < i_k$ is the lexicographically smallest index set. Define the \emph{reduced syzygy relations} $\widetilde \syzy(\eta_\bullet)$ to be the syzygy relations $\syzy(\widetilde \eta_\bullet)$ of this sub-tuple.
\end{definition}

Let $\gr(a,m)$ be the Grassmannian of $a$-dimensional subspaces in $\Ff_2^m$. 

\begin{definition}
Given $R \in \gr(m-k,m)$ and $\xa \in \Ff_2^{k+{k-1 \choose 2}}$ define
\begin{equation}
  \sgra \colonequals \{(X,\eta_\bullet) \in \sgt \mid R(\eta_\bullet) = R,\, \widetilde \syzy(\eta_\bullet) = \xa\}.
\end{equation}
If $R=0$ then we denote $\mathcal{S}^{(0,\mathfrak{a})}_g$ by $\sga$ and call it a \emph{non-degenerate component} of $\sgt$. When $R \neq 0$, the component $\sgra$ is \emph{degenerate}. 
\end{definition}

For any $R \in \gr(m-k,m)$ and $\xa \in \Ff_2^{k+{k-1 \choose 2}}$ we can consider two spaces $\sga \subset \sgtm{k}$ and $\sgra \subset \sgt$. 
It is clear that there is an isomorphism $\sga \isoto \sgra$ which maps $(X,\eta_\bullet)$ to $(X, \omega_X^{k_\bullet} \bigotimes_j \eta_j^{c_{\bullet,j}})$ where $k_i \in \Zz$ and $c_{i,j} \in \Ff_2$ depend only on $R$. 
In particular, $\sgra$ is non-empty iff $\sga$ is non-empty, regardless of $R$.

As a consequence, $\widetilde\syzy(\eta_\bullet)$ determines $\syzy(\eta_\bullet)$ and Theorem~\ref{thm:def_inv} holds with the syzygy relations $\syzy$ replaced by reduced syzygy relations $\widetilde\syzy$. 
Hence, the components of $\sgt$ are indexed by a subset of 
\begin{equation}\label{eq:index_set}
  \bigcup_{k=0}^{m-1} \gr(m-k,m) \times \Ff_2^{k+{k-1 \choose 2}}.   
\end{equation}

\begin{definition}
Let $\Syzy_g^m \subset \Ff_2^{m + {m-1 \choose 2}}$ be the set of all \emph{admissible syzygy relations},
\begin{equation}
  \Syzy_g^m = \{\xa \in \Ff_2^{m + {m-1 \choose 2}} \mid \sga \neq \emptyset\}.
\end{equation}
\end{definition}

Since $\sgra \simeq \sga$, the decomposition of $\sgt$ into its (non-empty) irreducible components is given by
\begin{equation}
  \sgt =  \bigcup_{k=0}^{m-1} \,\, \bigcup_{R \in \gr(m-k,m)} \,\, \bigcup_{\xa \in \Syzy_g^k} \sgra.   
\end{equation}
To complete the classification of components of $\sgm$, it remains to describe the set $\Syzy_g^m$ for any~$m$. 

\begin{theorem}\label{thm:extremes}
  If $g\ge m$ then $\Syzy_g^m = \Ff_2^{m + {m-1 \choose 2}}$, i.e., $\forall \xa \in \Ff_2^{m + {m-1 \choose 2}}$ the space $\sga$ is non-empty. If $g < \frac{m-1}{2}$ then $\Syzy_g^m = \emptyset$. In general, for given $(\xa,g)$, Algorithm~\ref{alg:syzygy_relations} decides if $\sga$ is empty or not. 
\end{theorem}
\begin{proof}
  We use the equivalence of theta characteristics and quadratic forms (Section~\ref{sec:theta_quadratic}) so that the problem is about the existence of a non-degenerate $m$-tuple of quadratic forms with syzygy relations $\xa$ on a $2g$-dimensional non-singular symplectic $\Ff_2$-vector space. This problem is addressed in Section~\ref{sec:existence}, in particular by Lemma~\ref{lem:inclusion_to_syzygy} and Proposition~\ref{prop:isometric-immersion}. See the last paragraph of Section~\ref{sec:existence} to see Algorithm~\ref{alg:syzygy_relations} rephrased in that context.
\end{proof}

\begin{remark}
  The bounds in the theorem are optimal. If $g < m$, some $\sga$ will be empty. In fact, the sequence $\xa$ consisting solely of $1$s will give an empty component whenever $g<m$. If $g \ge \frac{m-1}{2}$ then some $\sga$ will be non-empty. 
\end{remark}

\begin{example}
  For $g \ge 2$, $\sg^{2}$ has $4$ irreducible components: $\sg^{(0,0)},\sg^{(1,0)}, \sg^{(0,1)}, \sg^{(1,1)}$. If $g=1$ then $\sg^{(1,1)}$ is empty. Indeed, an elliptic curve has only one odd ($P=1$) theta characteristic. 
\end{example}

\begin{example}
  For $g \ge 3$, $\sg^{3}$ has $16$ components of the form $\sg^{(a_1,a_2,a_3; a_{23})}$. If $g=2$ then $\sg^{(1,1,1;1)}$ is empty. 
To see this directly, take a genus 2 hyperelliptic curve $C$ with Weierstrass points $w_1,\dots,w_6$. Up to relabeling, any three distinct odd theta characteristics can be written as $\eta_i \equiv w_i$ for $i=1,2,3$. But then $\eta_1 \otimes \eta_2\otimes \eta_3\inv\equiv w_1+w_2+w_3 - g_2^1$ is an even ($P = 0$) theta characteristic. 
\end{example}

\begin{algorithm}\label{alg:syzygy_relations}
  Given $\xa = (a_i;a_{ij}) \in \Ff_2^{m + {m-1 \choose 2}}$ and $g \in \Nn$ this algorithm decides if $\sga$ is empty or not. For $i,j \in \{2,\dots,m\}$ let $e_{ii}=0$ and $e_{ij} = a_{ij} + a_i +a_j +a_1$ for $i\neq j$. Let $E = \left( e_{ij} \right)_{2 \le i,j \le m}$ and $K = \ker E \subset \Ff_2^{m-1}$. Consider the quadratic form $q = \sum_{2 \le i \le m} (a_i+a_1)x_i^2 + \sum_{2 \le i < j \le m} e_{ij}x_ix_j$ as a function on $\Ff_2^{m-1}$ with coordinates $x_2,\dots,x_m$. Let $U \subset \Ff_2^{m-1}$ be a space complementary to $K$, i.e., $U \oplus K = \Ff_2^{m-1}$. If $q|_K \not\equiv 0$ or if the Arf invariant of $q|_U$ is $a_1$ then let $c=0$, otherwise let $c=1$. The equality $2g \ge (m-1) + \dim K + 2c$ holds iff $\sga$ is non-empty.
\end{algorithm}

For small $(m,g)$ we can count the number of non-degenerate components $\sga \subset \sgt$ using this algorithm. Let $N(m,g)$ be the cardinality of $\Syzy_g^m$. Theorem~\ref{thm:extremes} implies that $N(m,\ge m) = 2^{m+{m-1 \choose 2}}$ and $N(2k+1,<k)=0$. The Table~\ref{tab:num_comps} lists the values of $N(m,g)$ for $m,g \le 7$ obtained by applying Algorithm~\ref{alg:syzygy_relations} to each possible $\xa$.

\begin{table}[h!]
  \centering
\begin{tabular}{c|cccccccc}
\textbf{$m\backslash g$} & 0 & 1 & 2 & 3 & 4 & 5 & 6 & 7 \\ 
\hline
1 &  \multicolumn{1}{c|}{1} &  2 & 2 & 2 & 2 & 2 & 2 & 2 \\ \cline{3-3} \cline{2-2}
2 &  \multicolumn{1}{c|}{0} & \multicolumn{1}{c|}{3} & 4 & 4 & 4 & 4 & 4 & 4 \\ \cline{4-4}
3 &  \multicolumn{1}{c|}{0} & 4 & \multicolumn{1}{c|}{15} & 16 & 16 & 16 & 16 & 16 \\ \cline{5-5} \cline{3-3}
4 &  0 & \multicolumn{1}{c|}{0} & 84 & \multicolumn{1}{c|}{127} & 128 & 128 & 128 & 128 \\ \cline{6-6}
5 &  0 & \multicolumn{1}{c|}{0} & 448 & 1876 & \multicolumn{1}{c|}{2047} & 2048 & 2048 & 2048 \\ \cline{7-7} \cline{4-4}
6 &  0 & 0 & \multicolumn{1}{c|}{0} & 41664  & 64852 & \multicolumn{1}{c|}{65535} & 65536 & 65536 \\ \cline{8-8}
7 & 0 & 0 & \multicolumn{1}{c|}{0} & 888832 & 3819200 & 4191572 & \multicolumn{1}{c|}{4194303} & 4194304 
\end{tabular}
\caption{The number $N(m,g)$ of non-degenerate components of $\sgt$. If $g \ge m$ then $N(m,g) = 2^{m+{m-1 \choose 2}}$.  
If $2g < m-1$ then $N(m,g)=0$.}
  \label{tab:num_comps}
\end{table}

\begin{remark}
  It is easy to prove using Section~\ref{sec:existence} that $N(m,m-1) = N(m,m)-1$. The other extreme is also easy to describe: $N(2k+1,k)$ counts the number of non-singular quadratic forms on $\Ff_2^{2k}$. We can evaluate this number by summing the orbit sizes of an even and odd quadratic form: 
  \begin{equation}
    N(2k+1,k) = \frac{\#\gl(\Ff_2^{2k})}{\#O^+(2k)} + \frac{\#\gl(\Ff_2^{2k})}{\#O^-(2k)}, 
  \end{equation}
  where $O^{\pm}(2k) \subset \gl(\Ff_2^{2k})$ is the orthogonal group preserving a non-singular even ($+$) or odd ($-$) quadratic form. The cardinality of these groups are well known~\cite{kleidman90}:
  \begin{equation*}
    \#\gl(\Ff_2^{2k}) = \prod_{j=0}^{2k-1} \left(2^{2k} - 2^j \right),\quad 
    \#O^{\pm}(2k) = 2^{k^2-k+1} \cdot \left( 2^k \mp 1 \right) \prod_{j=1}^{k-1} \left( 2^{2j}-1 \right).  
  \end{equation*}
\end{remark}

\subsection{Theta characteristics and quadratic forms}\label{sec:theta_quadratic}

This section translates the set-up from the previous section on theta characteristics into the language of quadratic forms, which will be more convenient for the proofs. See~\cite{gross-harris--geometric_constructions} for our point of view and a survey of the connection between quadratic forms, theta characteristics and Arf invariants.

We recall our notation from the previous section: $X/\Cc$ is a smooth proper curve of genus $g \ge 0$ and $S_X$ is the set of theta characteristics on $X$, i.e., isomorphism classes of line bundles $\eta$ on such that $\eta^{\otimes 2} \simeq \omega_X$.

Let $V_X$ be the set of $2$-torsion line bundles of $X$ up to isomorphism, i.e., for $\varepsilon \in V_X$ we have $\varepsilon^{\otimes 2} \simeq \co_X$. We will make the identification $V_X = \H^1(X,\Ff_2)$ where the right hand side uses Betti cohomology. This identification comes from viewing $\pic^0_X$ as the quotient $\H^1(X,\co_X)/\H^1(X,\Zz)$ from which it follows that the $2$-torsion elements are $\H^1(X,\Ff_2)\simeq\H^1(X,\frac{1}{2}\Zz)/\H^1(X,\Zz)$. Note that $V_X \simeq \Ff_2^{2g}$ and $S_X$ is a $V_X$ torsor.

Using the intersection product on Betti cohomology (or equivalently the Weil pairing on $2$-torsion line bundles~\cite[p.282]{acgh}) we see that $V_X$ is a non-singular symplectic vector space: $\forall v \in V_X$, $v \cdot v = 0$ and $\forall v \in V_X,\, \exists u \in V_X$ such that $u \cdot v \neq 0$. Let $\spin(V_X)$ be the symplectic linear group acting on $V_X$. 

Denote by $Q_X \subset \sym^2 V_X^\vee$ the set of quadratic forms $V_X \to \Ff_2$ associated to the inner product on $V_X$. In characteristic $2$ this means that for any $q \in Q_X$ we have
\begin{equation}\label{eq:bilin}
  \forall u,v \in V_X, \quad q(u+v) + q(u) + q(v) = u \cdot v.
\end{equation}
The set $Q_X$ is an affine space (i.e., torsor) over $V_X$. Given $v \in V_X$ and $q \in Q_X$ we write $q+v$ for the quadratic form defined by $(q+v)(u) = q(u) + v\cdot u$. Conversely, given two quadratic forms $q,q' \in Q_X$ we denote by $q-q'$ the element of $V_X$ translating $q'$ to $q$. The only $\spin(V_X)$-invariant of a quadratic form $q \in Q_X$ is its Arf invariant $\arf(q) \in \Ff_2$, see~\cite{arf-invariant}. The value $\arf(q)$ coincides with the value which $q$ attains most frequently on $V_X$.   

The spaces $Q_X$ and $S_X$ are canonically identified, see~\cite[p.290]{acgh} and~\cite{mumford-theta-chars}. We will denote this identification by $S_X \isoto V_X \colon \eta \mapsto q_\eta$ where $q_\eta$ is defined by
\begin{equation}\label{eq:qeta}
 \forall v \in V_X, \quad  q_\eta(v) = h^0(\eta \otimes v) - h^0(\eta) \imod 2.
\end{equation}
The Arf invariant of the quadratic form $q_\eta$ coincides with the parity of $\eta$, \cite{mumford-theta-chars},
\begin{equation}
  \arf(q_\eta) = h^0(\eta) \imod 2 = P(\eta).
\end{equation}

As we are in characteristic $2$, we have $q_i-q_j = q_j -q_i$ so we define $q_i + q_j \colonequals q_i - q_j = q_j - q_i$. It is standard that with this operation $V_X \oplus Q_X$ becomes a vector space with $\Ff_2$-grading. 

\begin{definition}\label{def:Rq}
  A relation of $q_\bullet$ is an element $c_\bullet \in \Ff_2^m$ satisfying $\sum c_j q_j = 0$. The $\Ff_2$-grading implies that $\sum c_j = 0$. We denote the subspace of relations of $q_\bullet$ by $R(q_\bullet) \subset \Ff_2^m$. When $R(q_\bullet)=0$ we say $q_\bullet$ is non-degenerate. 
\end{definition}

Let $\eta_\bullet \in S_X^{m}$ be an $m$-tuple of theta characteristics and $q_\bullet \in Q_X^{m}$ be the corresponding sequence of quadratic forms ($q_i \colonequals q_{\eta_i}$).

\begin{lemma}\label{lem:Rq_is_Reta}
 The space of relations for $\eta_\bullet$ and $q_\bullet$ coincide, i.e., $R(q_\bullet) = R(\eta_\bullet) \subset \Ff_2^m$.
\end{lemma}
\begin{proof}
Using~\eqref{eq:qeta} and~\eqref{eq:bilin} it can be checked that $q_{i} + q_{j} = q_{\eta_i \otimes \eta_j \otimes \omega_X^\vee}$. It is now clear that the relations of $\eta_\bullet$ as defined in Definition~\ref{def:eta_relations} coincides with the relations of $q_\bullet$.
\end{proof}

\begin{remark}
Using $q_i + q_j = q_{\eta_i \otimes \eta_j \otimes \omega_X^\vee}$ we find 
 $ q_{1} + q_{2} + q_{3} = q_{\eta_1\otimes \eta_2 \otimes \eta_3 \otimes \omega_X^\vee}$.
Therefore, the syzygy relations of a sequence of theta characteristics on $X$ (Definition~\ref{def:syzygy}) and the syzygy relations of the corresponding quadratic forms (Definition~\ref{def:xa}) coincide.
\end{remark}

\begin{proposition}\label{prop:def_inv}
  For any deformation of $(X,\eta_\bullet)$ the space of relations $R(\eta_\bullet) \subset \Ff_2^m$ and the syzygy relations $\syzy(\eta_\bullet)$ remains unchanged. 
\end{proposition}
\begin{proof}
  As $(X,\eta_\bullet)$ varies in a smooth family, the space $V_X$ remains analytic locally constant (Ehresmann's fibration theorem) and, therefore, so does $Q_X$. As the parity of a theta characteristic is a deformation invariant~\cite{mumford-theta-chars} the identification of $\eta_\bullet$ with quadratic forms $q_\bullet$ can be done in families using~\eqref{eq:qeta}. In particular, the syzygy relations remain constant during deformations.

  As $V_X$ and $Q_X$ are discrete, the sections corresponding to $\eta_\bullet$ and $q_\bullet$ are locally constant. Then the space of relations $R(q_\bullet)$ is locally constant. But the ambient space $\Ff_2^m$ of $R(q_\bullet)$ is globally constant and then so is $R(q_\bullet)$. Now use Lemma~\ref{lem:Rq_is_Reta}. 
\end{proof}

\subsection{Monodromy action on theta characteristics}\label{sec:teich-spin}

We can deduce the connected components of $\sgt$ by studying the monodromy action of the finite cover $\pi \colon \sgt \to \mg$. The references for the standard results in this section are~\cite{earle-fowler--families_of_riemann_surfaces},~\cite{groth--teich} and~\cite[Appendix B]{acgh}. 

Let $X$ be a smooth proper complex curve of genus $g \ge 1$, $S_X$ the fiber of $\pi$ over $X$, $V_X = \H^1(X,\Ff_2) \simeq \Ff_2^{2g}$, and $Q_X$ is the affine space of quadratic forms on $V_X$ compatible with the symplectic intersection product.

We recall that the orbifold fundamental group of $\mg$ is isomorphic to the mapping class group $\Gamma_g$. The monodromy action of the mapping class group $\Gamma_g$ on $V_X = \H^1(X,\Ff_2)$ preserves the intersection product and induces a \emph{surjective} map $s\colon \Gamma_g \tos \spin(V_X)$. Let $\spin(V_X)$ act on $Q_X$ from the left by inverse precomposition, that is $\spin(V_X) \times Q_X \to Q_X : (\gamma,q) \mapsto q \circ \gamma\inv$. The following statement appears to be well known~\cite{sipe--roots-of-canonical} and, in any case, is not hard to prove (see~\cite[p.294]{acgh}).

\begin{proposition}\label{prop:monodromy}
  The monodromy action of $\Gamma_g$ on $S_X$ factors through the natural surjective map $s\colon \Gamma_g \tos \spin(V_X)$. With respect to the identification $S_X \mapsto Q_X \colon \eta \mapsto q_\eta$, the monodromy action on $S_X$ coincides with the precomposition action of $\spin(V_X)$ on $Q_X$. More precisely, $\forall u \in \Gamma_g,\, \forall \eta \in S_X$ we have $q_{u \cdot \eta} = q_\eta \circ s(u)\inv$.
\end{proposition}

Combining Proposition~\ref{prop:monodromy} with the irreducibility of $\mg$ we can determine the irreducible components of the $m$-fold product $\sgt = \sg \times_{\mg} \dots \times_{\mg} \sg$. 

\begin{corollary}\label{cor:monodromy-to-orbits}
  The irreducible (and connected) components of $\sgt$ are in bijection with the $\spin(V_X)$-orbits of $Q_X^{m} = Q_X \times \dots \times Q_X$ acting by precomposition, where $\gamma \in \spin(V_X)$ maps $(q_1,\dots,q_m)$ to $(q_1\circ \gamma\inv, \dots, q_m\circ \gamma\inv)$.
\end{corollary}

\begin{proof}
  The morphism $\pi \colon \sgt \to \mg$ is an orbifold finite cover. Therefore, the components of $\sgt$ are in bijection with the orbits of any fiber of $\pi$ under the monodromy action. Proposition~\ref{prop:monodromy} states that the monodromy action of $\pi_1(\mg,X) \simeq \Gamma_g$ on $S_X \simeq Q_X$ factors through a surjective map $\Gamma_g \tos \spin(V_X)$ followed by the precomposition action of $\spin(V_X)$ on $Q_X$.
\end{proof}

\subsection{Affine geometry of quadratic forms}\label{sec:affine-geometry}

The goal of this subsection is to classify the orbits of $m$-tuples of quadratic forms on a symplectic $\Ff_2$-vector space under the action of the symplectic group. Via Corollary~\ref{cor:monodromy-to-orbits} this will allow us to classify the components of $\sgt$.  The standard text book~\cite{artin--geometric_algebra} offers a comprehensive treatment of quadratic spaces in characteristic not $2$, the basic treatment in characteristic $2$ is similar~\cite{knebusch2010}. 

Let $V$ be a $2g$-dimensional vector space over $\Ff_2$ with a non-singular symplectic inner product $f\colon V\times V \to \Ff_2$. Let $Q$ be the set of quadratic forms on $V$ associated to the inner product $f$.  The set $Q$ is an affine space with space of translations $V$. The translation action is given by $V \times Q \to V\colon (v,q) \mapsto q + f(v,\cdot)$. We will simply write the latter quadratic form as $q+v$. Moreover, if $q,q' \in Q$ are such that $q' = q +v$ then we will express $v$ as $q'-q$, or as $q'+q$ since the characteristic is 2. 

Let $\spin(V)\colonequals \spin(V,f)$ be the symplectic group of $V$ preserving $f$. There is a natural action of $\spin(V)$ on $Q$ via precomposition. In other words, there is a pairing $\spin(V)\times Q \to Q : (\gamma,q) \mapsto \gamma_*q = q\circ \gamma\inv$. Notice that $\gamma_*(q+v)=\gamma_*q + \gamma(v)$ since $f$ is invariant under $\gamma$. The Arf invariant $\arf \colon Q \to \Ff_2$ is the only invariant under the action of $\spin(V)$~\cite{arf-invariant}. 

To classify the $\spin(V)$-orbits of $Q^{m} = Q \times \dots \times Q$ under precomposition we could remove the diagonals from $Q^{m}$ and use induction, since the orbits contained in a diagonal correspond to orbits in $Q^{m-1}$. However, we can do better: the affine linear structure on $Q$ allows us to reduce the number of elements in a tuple of quadratic forms to an affine linearly independent subset. As a consequence, we need only consider affine linearly independent tuples from $Q^m$.

\begin{definition}\label{def:degenerate-tuple}
  A sequence $q_\bullet = (q_1,\dots,q_m)$ of quadratic forms on $V$ is called \emph{non-degenerate} if the affine span generated by $q_\bullet$ in $Q$ is of dimension $m-1$. (This definition agrees with the one in Definition~\ref{def:Rq}.)
\end{definition}

For non-degenerate tuples, we will see that the ``syzygy relations'' as defined below classify the $\spin(V)$-orbits.

\begin{definition}\label{def:xa}
  Given a sequence of quadratic forms $q_\bullet=(q_1,\dots,q_m)$, let the \emph{syzygy relation of $q_\bullet$} to be the tuple 
  \begin{equation}
    \syzy(q_\bullet) \colonequals ( \arf(q_i) ; \arf(q_i + q_j + q_1))_{\substack{1 \le i \le m \\ 1 < j < k \le m}} \in \Ff_2^{m + {m-1 \choose 2}}. 
  \end{equation}
\end{definition}

\begin{notation}\label{not:Wq}
Given $q_\bullet \in Q^{m}$ define $W(q_\bullet) = (W,q)$ where $W$ is obtained by translating the affine span of $q_\bullet$ by $q_1$ and the quadratic form $q$ is the restriction of $q_1$ to $W \subset V$. The obvious generating vectors for $W$ will be denoted by $v_i = q_i+q_1$ for $i=2,\dots,m$. 
Let $e_i = q(v_i)$, $e_{ij} = f(v_i,v_j)$. 
\end{notation}

\begin{lemma}\label{lem:xe}
 If $\syzy(q_\bullet) = (a_i;a_{ij})$ then $q(v_i) = a_i + a_1$ and $f(v_i,v_j) = a_{ij}+a_i + a_j + a_1$. 
\end{lemma}
\begin{proof}
  There is a relation between the Arf invariant of a quadratic form and the Arf invariant of its translate, which give 
  \begin{equation}
    \arf(q_{i}) = \arf(q_1+v_i) = \arf(q_1) + q_1(v_i).
  \end{equation}
  Hence we deduce $q(v_i)= a_1+a_{i}$. To evaluate $f(v_i,v_j)$ we combine the relation above and the bilinearity relation as in~\eqref{eq:bilin}.
\end{proof}

\begin{witt}[Theorem~3.3 and Exercise~3.31 in~\cite{wilson--finite_simple_groups}] \label{lem:witt}
  Let $(V,q)$ and $(V',q')$ be isometric non-singular quadratic spaces. Let $W \subset V$ and $W' \subset V'$ be subspaces and $\mu\colon (W,q|_W) \isoto (W',q'|_{W'})$ be an isometry. Then, there is an isometry $\gamma \colon  (V,q) \to (V',q')$ such that $\gamma|_W = \mu$.
\end{witt}

\begin{proposition}\label{prop:orbit-of-quadratics}
  Two sequences of non-degenerate quadratic forms on $V$ are in the same $\spin(V)$-orbit if and only if their syzygy relations are equal.
\end{proposition}

\begin{proof}
  Let $q_\bullet, q'_\bullet \in Q^m$ be non-degenerate. 
   Since $\spin(V)$ preserves the Arf invariant, if there is a $\gamma \in \spin(V)$ such that $\gamma_*q_\bullet = q'_\bullet$ then the associated syzygy relations are equal, that is, $\syzy(q_\bullet)=\syzy(q'_\bullet)$. 

   Conversely, suppose that $\syzy(q_\bullet)=\syzy(q_\bullet')$. Since, $\arf(q_1)=\arf(q_1')$ the non-singular quadratic spaces $(V,q_1)$ and $(V,q_1')$ are isometric. Let $W(q_\bullet) = (W,q)$ and $W(q_\bullet') = (W',q')$ be given as in Notation~\ref{not:Wq}. The obvious generating elements form a basis since $q_\bullet$ and $q_\bullet'$ are non-degenerate. We denote these bases by $W = \langle v_2,\dots,v_m \rangle$ and $W' = \langle v_2',\dots,v_m' \rangle$. 

   Define a linear map $\mu\colon W \to W' : v_i \mapsto v'_i$. Lemma~\ref{lem:xe} and the equality of the syzygy relations implies that $q(v_i)=q'(v'_i)$ and $f(v_i,v_j) = f(v_i',v_j')$. Therefore, $\mu\colon(W,q_1|_{W}) \to (W',q'_1|_{W'})$ is an isometry.

Applying Witt's Lemma above we conclude that there is an isometry $\gamma\colon(V,q_1) \to (V,q_1')$ extending $\mu$. Necessarily $\gamma \in \spin(V)$ and $\gamma_*q_1 = q_1'$. Furthermore, $\gamma_*(q_i) = \gamma_*(q_1 + v_i) = \gamma_*(q_1) + \gamma(v_i) = q_i'$. Thus $\gamma_*q_\bullet = q_\bullet'$.
\end{proof}

Let $q_\bullet \in Q^m$ be a possibly degenerate $m$-tuple. We will use the space of relations $R(q_\bullet) \subset \Ff_2^m$ from Definition~\ref{def:Rq} and the obvious adaptation of the reduced syzygy relations $\widetilde\syzy(q_\bullet)$ from Definition~\ref{def:reduced_syzygy}. The following follows immediately from Proposition~\ref{prop:orbit-of-quadratics}.

\begin{corollary}\label{cor:degenerate_orbit}
  Two tuples $q_\bullet, q'_\bullet \in Q^m$ are in the same $\spin(V)$-orbit if and only if $R(q_\bullet) = R(q'_\bullet)$ and $\widetilde\syzy(q_\bullet) = \widetilde\syzy(q'_\bullet)$.
\end{corollary}

\subsection{Prescribing syzygy relations}\label{sec:existence}

In this section we will give an algorithm to decide if given $\xa = (a_i;a_{ij}) \in \Ff_2^{m+{m-1 \choose 2}}$ there is a non-degenerate sequence $q_\bullet \in Q^m$ with $\syzy(q_\bullet) = \xa$. 

With Notation~\ref{not:Wq} and Lemma~\ref{lem:xe} in mind we define what would be an abstract copy of $W(q_\bullet)$ if non-degenerate $q_\bullet$ corresponding to $\xa$ exists.

\begin{notation}\label{not:wxa}
  Let $(W_\xa,q_\xa)$ be such that $W_\xa=\Ff_2^{m-1}$ with basis $v_2,\dots,v_m$ and dual basis $x_2,\dots,x_m$. Let 
  \begin{equation}\label{eq:qxa}
  q_\xa = \sum_{i=2}^m e_ix_i^2 + \sum_{2 \le i < j \le m} e_{ij} x_ix_j,
\end{equation}
where $e_i=a_i+a_1$ and $e_{ij} = a_{ij} + a_i + a_j + a_1$. 
\end{notation}

\begin{lemma}\label{lem:inclusion_to_syzygy}
  There exists a non-degenerate sequence $q_\bullet$ of quadratic forms on $V$ with syzygy relations $\xa$ if and only there is a quadratic form $q_V$ on $V$ with $\arf(q_V)=a_1$ and an isometric immersion $\iota\colon(W_\xa,q_\xa) \toi (V,q_V)$.
\end{lemma}
\begin{proof}
  If there is an immersion $\iota$, let $q_1=q_V$ and $q_i = q_1 + \iota(v_i)$ for $i=2,\dots,m$. Then $q_\bullet=(q_1,\dots,q_m)$ is non-degenerate with syzygy relations $\xa$. Conversely, given $q_\bullet = (q_1,\dots,q_m)$ let $q_V = q_1$ and consider the subspace $W(q_\bullet)$ of $V$ as defined in Notation~\ref{not:Wq}. By design $(W_\xa,q_\xa)$ is isomorphic to $W(q_\bullet)$, giving the map $\iota$. 
\end{proof}

We now give an easy criterion to check for the existence of immersions of the kind required by Lemma~\ref{lem:inclusion_to_syzygy}. Let $(V,q_V)$ and $(W,q_W)$ be quadratic spaces with $V$ non-singular symplectic and $W$ possibly singular symplectic. The intersection pairing on $W$ defines a linear map $W \to W^\vee$ with kernel $K$. Choose a subspace $U \subset W$ complementary to $K$, i.e., $W = K \oplus U$. Note that the intersection pairing on $U$ is non-degenerate and the value of $\arf(q_W|_U)$ does not depend on our choice of $U$. 

\begin{notation}\label{not:correction-term}
  For $a\in \Ff_2$ let $c(q_W,a) \in \set{0,1}$ be the \emph{correction term} defined as follows:
  \[
    c(q_W,a) = \left\{ \begin{array}[]{ccc}
      0 &:&   q_W|_K \not\equiv 0  \\
      0 &:&   q_W|_K \equiv 0,\,\arf(q_W|_U) = a \\
      1 &:&   q_W|_K \equiv 0,\,\arf(q_W|_U) \neq a 
    \end{array} \right.
  \]
\end{notation}

\begin{remark}
  In evaluating the correction term for $K=0$ or $U=0$, use the convention that the quadratic form on the zero space is zero and has Arf invariant zero.
\end{remark}

\begin{proposition}\label{prop:isometric-immersion}
  There is an isometric immersion $(W,q_W) \toi (V,q_V)$ iff the following inequality is satisfied:
  \begin{equation}\label{eq:inequality}
    \dim V \ge \dim W + \dim K +2c(q_W,\arf(q_V)).
  \end{equation}
\end{proposition}
\begin{proof}
  Recall our decomposition $W = K \oplus U$. Let $K'$ be an isomorphic copy of $K$ and define $W' = W \oplus K'$. Extend the symplectic pairing on $W$ to $W'$ so that $K'$ is orthogonal to $U$ and is dual to $K$. Naturally, $W'$ is a non-singular symplectic space and any embedding $W \toi V$ will extend to an embedding $W' \toi V$. This forces the inequality $\dim V \ge \dim W + \dim K$.

  Pick a basis $\gamma_1,\dots,\gamma_s$ of $K$ and its dual basis $\gamma_1',\dots,\gamma_s' \in K'$. Any extension of the quadratic form $q_W$ to a quadratic form $q_{W'}$ on $W'$ requires only the values $q_{W'}(\gamma_i')$ for $i=1,\dots,s$. Using the obvious hyperbolic decomposition of $K\oplus K'$ defined by our choice of bases, we note that:
  \[
    \arf(q_{W'}) = \arf(q_W|_U) + \sum_{i=1}^{s}q_{W}(\gamma_i)q_{W'}(\gamma_i').  
  \]
  If $q_W|_K \not\equiv 0$ then we can choose an extension $q_{W'}$ of either Arf invariant. However, if $q_W|_K \equiv 0$ then any extension $q_{W'}$ will necessarily have $\arf(q_{W'})= \arf(q_W|_U)$. To change the parity, we would have to join an odd plane to $W'$, increasing the dimension by 2 and obtaining, say $(W'',q_{W''})$. This forces the refined inequality $\dim V \ge \dim W + \dim K + 2c(q_W,\arf(q_V))$.

  However, this inequality is also sufficient. Construct $(W',q_{W'})$ (or the larger $(W'',q_{W''})$ if necessary) such that $\arf(q_{W'}) = \arf(q_V)$ (or $\arf(q_{W''})=\arf(q_V)$). Using the hyperbolic decompositions of $V$ and $W'$ (or $W''$) it is clear that we can find an embedding $W' \toi V$ (or $W'' \toi V$).
\end{proof}

\begin{corollary}\label{cor:isometric-immersion}
  For any $\xa \in \Ff_2^{m+{m-1 \choose 2}}$ there exists a non-degenerate sequence $q_\bullet=(q_1,\dots,q_m)$ of quadratic forms on $V$ with syzygy relations $\xa$ if 
\begin{equation}  \label{eq:g_m}
  g \ge  m, \quad\quad \text{where } g = \frac{\dim V}{2}.
\end{equation}
\end{corollary}
\begin{proof}
  Observe that $m-1 = \dim W_\xa \ge \dim K $ and $1 \ge c(q_\xa,\bullet)$ so that the right hand side of~\eqref{eq:inequality} is at most $2m$. Since $\dim V \ge 2m$, Proposition~\ref{prop:isometric-immersion} implies the existence of an isometric immersion $(W_\xa,q_\xa) \toi (V,q_V)$. We conclude by Lemma~\ref{lem:inclusion_to_syzygy}.
\end{proof}

\begin{remark}\label{rem:optimal}
  Note that the inequality~\eqref{eq:g_m} is optimal. Indeed, if $g < m$ then the tuple $\xa$ consisting entirely of $1$s does not constitute a syzygy relation. In this case, the intersection product on $W_\xa$ is identically zero. Therefore, $K=W$, $U=0$, and $\arf(q_\xa|_U)=0 \neq 1 = \arf(q_V)$. The inequality~\eqref{eq:inequality} is then violated.
\end{remark}

In general, we can decide if $\xa$ is the syzygy relation of a non-degenerate $m$-tuple of quadratic forms on a $2g$-dimensional non-singular symplectic space as follows. The tuple $\xa$ allows us to construct $q_\xa$ as in~\eqref{eq:qxa}. The bilinear pairing on $W_\xa = \Ff_2^{m-1}$ induced by $q_\xa$ is given by the coefficients $e_{ij}$ of $x_ix_j$ ($i \neq j$). In the context of Lemma~\ref{lem:inclusion_to_syzygy} we can thus evaluate the correction term $c(q_\xa,a_1)$ to check if the inequality in Proposition~\ref{prop:isometric-immersion} holds. In this setting, this inequality specializes to
\begin{equation}
  2g \ge (m-1) + \dim K + c(q_\xa,a_1).
\end{equation}
The inequality holds if and only if the syzygy relations are realized. This is Algorithm~\ref{alg:syzygy_relations}.

\appendix
\section{The universal deformation of a rooted node}\label{app:local_deformations}

In this section we describe the universal deformation of a node together with a root of a line bundle, that is, of a \emph{rooted node}. This amounts to bringing together the results available in literature, specifically~\cite{faltings-bundles} and~\cite{jarvis-torsion-free}. Faltings~\cite{faltings-bundles} studies, essentially, square roots of vector bundles whereas Jarvis~\cite{jarvis-torsion-free} studies $r$-th roots of line bundles. We are interested in the intersection of the two, the square roots of line bundles. Because of its simplicity, a treatment of this special case is quite revealing.

\subsection{Conventions}

In this section we are solely concerned with local, or inifinitesimal, deformation functors. Therefore, at times, we will simply refer to these as deformation functors. Infinitesimal deformations of an affine scheme are always affine~\cite{sernesi-deformation} and therefore we will work in the dual category of algebras instead of schemes. 

For the rest of this section $k$ is an algebraically closed field of characteristic $\neq 2$ and  $\Lambda$ is a complete noetherian local ring with residue field $k$. The category of Artinian local $\Lambda$-algebras with residue field $k$ is denoted by $\art_\Lambda$. We denote by $\hart_\Lambda$ the category of complete noetherian local $\Lambda$-algebras $(R,\xm_R)$ such that for each $n \ge 1$ the quotient $R/\xm_R^n$ belongs to $\art_\Lambda$. Every $R \in \hart_\Lambda$ comes equipped with a natural map $R \to k = R/\xm_R$.

A functor $F\colon \art_\Lambda \to \Sets$ is said to be \emph{pro-represented} by $R \in \hart_\Lambda$ if $F$ is isomorphic to the restriction of the functor $\hom_{\hart_{\Lambda}}(R,\cdot)$ to $\art_{\Lambda}$, see~\cite{schlessinger} for more details. 

We will use equality between two objects $X=Y$ to mean that there exists a unique isomorphism between~$X$ and~$Y$.

\subsection{Deformations of a node}

\begin{definition}\label{def:std_node}
  The $k$-algebra $\bar A \colonequals  k[[x,y]]/(xy)$ is called \emph{the standard node}. A tuple $(A/ R , \iota\colon A \to \bar A)$ where $A$ is a complete local flat $R$-algebra and $\iota$ factors through an isomorphism $A\otimes_R k \to \bar A$ is called \emph{a deformation of the node over $R$}. Two deformations $(A/ R,\iota)$ and $(A' / R,\iota')$ are isomorphic if there is an $R$-isomorphism of $A$ and $A'$ commuting with the maps $\iota$ and $\iota'$.
\end{definition}

\begin{definition}\label{def:G}
  The \emph{functor of infinitesimal deformations of the node} is the functor $G \colon \art_\Lambda \to \Sets$ which maps $R$ to the set of isomorphism classes of deformations of the node over $R$. 
\end{definition}

\noindent The following theorem is folklore. See \stacks{0CBX} for the idea of proof.

\begin{theorem} \label{thm:univ-def-of-node}
  The deformation $(\Lambda[[x,y,t]]/(xy -t) / \Lambda[[t]],j\colon t \mapsto 0)$ is universal. That is, $\Lambda[[\tau]]$ pro-represents $G$, which for any deformation $(A / R,\iota) \in G(R)$ defines a unique map $\Lambda[[t]] \to R$ giving rise to an isomorphism $A \simeq R[[x,y]]/(xy-\pi)$.
\end{theorem}

\subsection{Deformations of a root}

\begin{definition}
  Let $(A/ R,\iota)$ be a deformation of the node and let $E$ be an $R$-flat and $R$-relatively torsion-free rank-1 $A$-module. Suppose $b\colon E^{\otimes 2} \to A$ is a non-degenerate bilinear form (Definition~\ref{def:bilinear}). Then the pair $(E,b)$ will be called a \emph{root}. If $E$ is a free $A$-module, and hence $E \simeq A$, then $(E,b)$ is called a \emph{free root}. An isomorphism between two roots $(E,b)$ and $(E',b')$ is an isomorphism $\mu \colon E \to E'$ such that $b' \circ \mu^{2} = b$. 
\end{definition}

\begin{notation}\label{not:iota_push}
  For a root $(E,b)$ on $(A / R,\iota)$ we will write $\iota_* E$ for $E \otimes_A \bar A$ and $\iota_*b$ for the map $\iota_*E^{2} \to \bar A$ induced from $b$.
\end{notation}

\begin{definition}
  Let $(\bar E, \bar b)$ be a root on the standard node and let $(E,b)$ be a root on a deformation of the node $(A / R, \iota)$. An isomorphism $j \colon  \iota_* E \iso \bar E$ such that $\bar b \circ j^{2} = \iota_* b$ is called a \emph{restriction map}. The tuple $(E,b,j)$ is a \emph{deformation of the root $(\bar E, \bar b)$}. An isomorphism of deformations is an isomorphism of roots commuting with the restriction maps.
\end{definition}

\begin{definition}
  The standard node together with a root $(\bar A / k, \bar E,\bar s)$ will be called a \emph{rooted node}. A deformation of the standard node together with a deformation of the root $(\bar E, \bar s)$ is a tuple $(A / R, \iota , (E, b, j))$, which will be called a \emph{deformation of the rooted node over $R$}. Isomorphisms are defined in the obvious way.
\end{definition}

\begin{lemma} \label{lem:free_roots_deform_trivially}
  If $(\bar E,\bar s)$ is a free root on $\bar A / k$, then for any $R \in \hart_{\Lambda}$ there exists a unique deformation of $(\bar A/ k,\bar E, \bar s)$ over $R$, up to unique isomorphisms.
\end{lemma}

\begin{proof}
  Since $A$ is complete with respect to $\xm_R\cdot A$ we just have to show that there exists a unique lift of a free root from $R/\xm_R^n$ to $R/\xm_R^{n+1}$. Existence of the lift is clear. In order to conclude that there exists a \emph{unique} isomorphisms between any two lifts, we observe that any two lifts of square roots of an invertible element in $R/\xm_R^n$ are equal.
\end{proof}
\begin{remark}\label{rem:free_roots_deform_trivially_always}
  The proof does not require the presence of a node and the argument works just as well around a smooth point of a curve.
\end{remark}

\begin{definition}\label{def:deformations_of_rooted_node}
  Given a root $(\bar E,\bar s)$ on $\bar A / k$ let $F\colon \art_\Lambda \to \Sets$ be the functor of isomorphism classes of deformations of the rooted node $(\bar A / k, \bar E, \bar s)$. 
\end{definition}

Recall that $G\colon \art_{\Lambda} \to \Sets$ is the functor of infinitesimal deformations of the node $\bar A / k$. There is a natural transformation $F \to G$ obtained by forgetting the root. Lemma~\ref{lem:free_roots_deform_trivially} implies that if $(\bar E,\bar s)$ is a free root then the forgetful functor $F \to G$ is an isomorphism. Complementing Theorem~\ref{thm:univ-def-of-node} we have the following result.

\begin{theorem}
  If $(\bar E,\bar s)$ is not free then $F$ is pro-represented by $\Lambda[[\tau]]$. The natural transformation $F \to G$ obtained by forgetting the root corresponds to the map $\Lambda[[t]] \to \Lambda[[\tau]] : t \mapsto \tau^2$.
\end{theorem}

We will prove this theorem by constructing a universal family over $\Lambda[[\tau]]$, see Theorem~\ref{thm:univ-def-of-rooted-node}. To do this we need Faltings' classification of torsion-free modules.

\subsection{Classification of roots}

Fix $R \in \hart_\Lambda$ and $A = R[[x,y]]/(xy-\pi)$ for some $\pi \in \xm_R$. Define $\iota\colon A \to \bar A = k[[x,y]]/(xy)$ using $R \to R/\xm_R = k$. For the rest of this section, we will be interested in non-free roots.

\subsubsection{Faltings' construction} \label{sec:faltings-construction}

Let $p,q \in R$ be such that $pq = \pi$. Define $2\times 2$ matrices with entries in $A$: 
\[
  \alpha =  \begin{pmatrix} x & p \\  q & y \end{pmatrix}, \quad \beta = \begin{pmatrix} y & -p \\  -q & x \end{pmatrix}
\]
To avoid confusion, we may write $\alpha(p,q)$ and $\beta(p,q)$ instead. Clearly $\alpha \beta = \beta \alpha = 0$ but, moreover, we get an exact infinite periodic complex (see~\cite{faltings-bundles}):
\[
  \ldots \to A^{\oplus 2} \overset{\alpha}{\to} A^{\oplus 2} \overset{\beta}{\to} A^{\oplus 2} \overset{\alpha}{\to} A^{\oplus 2} \overset{\beta}{\to} A^{\oplus 2} \to \ldots
\]

\begin{definition}\label{def:std_res}
  Define $E(p,q)\subset A^{\oplus 2}$ to be the image of $\alpha$ or, equivalently, the kernel of $\beta$. Truncating the complex above we get a free resolution of $E(p,q)$. Whenever we refer to the standard resolution of $E(p,q)$ this is the one we mean.
\end{definition}

\begin{remark}
  The module $E(p,q)$ is relatively torsion-free and $R$-flat as shown in Construction 3.2 of~\cite{faltings-bundles}.
\end{remark}

\begin{remark}
  If either $p$ or $q$ is invertible, then $E(p,q)$ is free. As we are interested in non-free roots, from now on we assume $p,q \in \xm_R$. This forces $\pi$ to be in $\xm_R^2$.
\end{remark}

There is a natural isomorphism between $E(p,q)^\vee = \hom(E(p,q),A)$ and $E(q,p)$. In particular, when $p=q$ the module $E(p,p)$ is self-dual and we get a natural pairing $s_p \colon E(p,p)^2 \to A$.

\begin{definition}
  The natural pairing $s_p \colon E(p,p)^{2} \to A$ will be called the \emph{standard map}. 
\end{definition}

\begin{definition}\label{def:std_root}
  Let us refer to $(E(p,p),s_p)$ as a \emph{standard root on $(A / R,\iota)$}.
\end{definition}

\subsubsection{Properties of standard roots}

Given any bilinear map $b$ on $E(p,q)$ we can lift it to $\sym^2 A^{\oplus 2} \tos \sym^2 E(p,q)$ to get a morphism $\tilde b : \sym^2 A^{\oplus 2} \to A$. Letting $e_1,e_2$ be the standard generators of $A^{\oplus 2}$ and $e_1^2, e_1e_2, e_2^2 $ the corresponding generators of $\sym^2 A^{\oplus 2}$ we may uniquely identify $b$ with the values $b_0 \colonequals  \tilde b(e_1^2), b_1 \colonequals  \tilde b(e_1e_2), b_2 \colonequals  \tilde b(e_2^2)$. By abuse of notation we will write $b = (b_0,b_1,b_2)$. 

\begin{lemma}
  For a standard root $(E(p,p),s_p)$ we have $s_p=(x,p,y)$.
\end{lemma}
\begin{proof}
  Let $\langle \cdot,\cdot \rangle$ denote the natural pairing $A^{\oplus 2} \times A^{\oplus 2} \to A$. The identification of $E(p,p)^\vee$ with $E(p,p)$ makes it clear that if $e,f \in E(p,p)$ and $u,v \in A^{\oplus 2}$ are such that $e = \alpha(u)$ and $f=\alpha(v)$ then we have $s_p(e,f) = \langle u, \alpha(v) \rangle = \langle \alpha(u), v \rangle$.  Now, direct computation yields the result.
\end{proof}

\begin{lemma} \label{lem:scale-b}
  Any root $(E(p,p),b)$ on $A / R$ is isomorphic to $(E(p,p),s_p)$.
\end{lemma}
\begin{proof}
  Lemma 5.4.10~\cite{jarvis-torsion-free} states that $b = (ax,b_1,awy)$ where $a \in A^*$ and $w \in R^*$ such that $wp = p$.  Note here that as we are working with square roots of line bundles, the hypothesis of the cited lemma is satisfied (as stated in Corollary 5.4.9 \emph{loc.cit.}).

  Let $v$ be a square root of $w$ and consider the isomorphism $\mu : E \to E$ which descends from multiplication by $\begin{psmallmatrix} v & 0 \\   0 & v\inv  \end{psmallmatrix}$ on $A^{\oplus 2}$. Clearly $\mu^*b = a(vb_0,b_1,v\inv b_2) = va(x,v\inv b_1, y)$. By scaling $E$ we may now assume $va = 1$ and $b = (x,b_1,y)$ where we changed $b_1$.

  Since $\alpha(y,0)=\alpha(0,p)$ and $\alpha(0,x)=\alpha(p,0)$ we see that $pb_2 = y b_1$ and $pb_0= xb_1$. Which means $y(b_1-p) = x(b_1 -p) = 0$ (we used $wp =p$). But $\ann_A(x,y) = 0$ hence $b_1 = p$. 
\end{proof}

\begin{theorem}[Faltings] \label{thm:faltings}
  For any root $(E,b)$ on $A$, $\exists\, p \in \xm_R$ such that $(E,b) \iso (E(p,p),s_p)$.
\end{theorem}
\begin{proof}
In~\cite[Theorem 3.7]{faltings-bundles} Faltings classifies non-degenerate quadratic forms on $E$, or equivalently bilinear forms $b \colon  E^{2} \to A$. The conclusion is that $(E,b) \simeq (E(p,p),b')$ for some $p \in \xm_R$ and $b'$. Apply Lemma~\ref{lem:scale-b} to change $b'$ to the standard map $s_p$. 
\end{proof}

\begin{notation}
  If $\mu \in \hom(E(p,p),E(q,q))$ lifts to a map $\begin{psmallmatrix} a_{11} & a_{12}  \\  a_{21} & a_{22}  \end{psmallmatrix} : A^{\oplus 2} \to A^{\oplus 2}$ then we will write $\mu = \begin{bsmallmatrix} a_{11} & a_{12}  \\  a_{21} & a_{22}  \end{bsmallmatrix}$.
\end{notation}

\begin{lemma} \label{lem:iso-class}
$\isom((E(p,p),s_p),(E(q,q),s_q)) = \left\{\begin{bsmallmatrix} \varepsilon_1 & 0  \\  0 & \varepsilon_2  \end{bsmallmatrix} \mid \varepsilon_1,\,\varepsilon_2 \in \set{\pm 1},\, q = \varepsilon_1 \varepsilon_2 p\right\}$ 
\end{lemma}

\begin{proof}
  We adapt the proof of Proposition 4.1.12 of~\cite{jarvis-torsion-free}. An easy observation is that we can choose a lift of $\mu$ of the form  
  \[
    \begin{pmatrix} u_+(x) & v_+(x)  \\  v_-(y) & u_-(y) \end{pmatrix}
  \]
  where $u_+, v_+ \in R[[x]] \subset A$ and $u_-,v_- \in R[[y]] \subset A$. Now we simply have to calculate what it means to have $\mu^* s_p = s_q$ in terms of $u_{\pm}, v_{\pm}$. Using that $x$ (resp. $y$) does not annihilate $R[[x]]$ (resp. $R[[y]]$) we see immediately that $v_{\pm }=0$, $u_{\pm} \in \set{\pm 1}$ is forced. Then $q = u_+ u_- p$.
\end{proof}

\begin{notation}
  On the standard node $\bar A/k$ we will define $\bar E\colonequals E(0,0)$ and $\bar s \colonequals  s_0$.
\end{notation}

\begin{definition}\label{def:restriction_map}
  On $(A/ R,\iota)$ there is a \emph{natural restriction map} from $E(p,q)$ to $\bar E = E(0,0)$ which is the map $r$ completing the diagram below:
\[
  \begin{tikzcd}
    A^{\oplus 2} \arrow[d,"{\alpha(p,q)}"] \arrow[r,"\iota"]  & \bar A^{\oplus 2} \arrow[d,"{\alpha(0,0)}"] \\
    E(p,q) \arrow[r, dotted, "r"] & \bar E
  \end{tikzcd}
\]
\end{definition}

\begin{remark} \label{rem:unique-iso}
  Lemma~\ref{lem:iso-class} implies that choosing a restriction map $r$ rigidifies the root. That is, $\aut(E(p,p),s_p,r) = 1$. The following result takes this observation one step further.
\end{remark}

\begin{proposition} \label{prop:all-deformations-are-standard}
  Suppose that $(E,b,j)$ is a deformation of $(\bar E, \bar s)$. Then there exists precisely one $p \in R$ such that $(E,b,j) \simeq (E(p,p),s_p,r)$. Moreover, this isomorphism is unique.
\end{proposition}
\begin{proof}
  Uniqueness of the isomorphism follows from Remark~\ref{rem:unique-iso}. By Theorem~\ref{thm:faltings} we know that $(E,b) \simeq (E(p,p),s_p)$ for some $p\in R$. Picking one such isomorphism we may assume $(E,b,j) = (E(p,p),s_p,j)$ for some $j$. However, with our choice of identification, $j$ is not necessarily equal to the natural restriction $r$.  

  Let $\gamma = j \circ r\inv : (\bar E,\bar s) \to (\bar E, \bar s)$. Then $\gamma$ is uniquely defined by $\gamma(0) \in \set{\pm \id, \pm \begin{psmallmatrix} 1 & 0  \\  0 & -1  \end{psmallmatrix}}$. An isomorphism $\mu \colon  (E(p,p),s_p) \iso (E(q,q),s_q)$ commutes with $j$ and $r$ iff $\iota_*\mu \colon (\bar E, \bar s) \iso (\bar E, \bar s)$ is the inverse of $\gamma$. Having classified such $\mu$ in Lemma~\ref{lem:iso-class} we know that there exists precisely one $q$ and one $\mu$ which will restrict to $\gamma\inv$.
\end{proof}

\subsection{Universal deformation of rooted node}

Take a \emph{non-free} root $(\bar E,\bar s)$ on $\bar A / k$ and let $F\colon \art_\Lambda \to \Sets$ be the deformation functor of $(\bar A / k, \bar E, \bar s)$ (Definition~\ref{def:deformations_of_rooted_node}). 

\begin{theorem} \label{thm:univ-def-of-rooted-node}
  The ring $\Lambda[[\tau]]$ pro-represents $F$ via the universal family 
  \[
    ( \Lambda[[\tau]] \to \Lambda[[x,y,\tau]]/(xy-\tau^2) , \tau\mapsto 0, E(\tau,\tau), s_\tau,r).
  \]
\end{theorem}

\begin{proof}
  Given any deformation of the rooted node $(A / R, \iota, E ,b,j)$ we wish to show that there exists a unique map $\varphi \colon \Lambda[[\tau]] \to R$ such that $A$ is canonically isomorphic to $\Lambda[[x,y,\tau]]/(xy-\tau^2) \otimes_{\Lambda[[\tau]]} R$ and $\varphi_*(E(\tau,\tau),s_\tau,r) \simeq (E,b,j)$. Furthermore, that this isomorphism is unique.

  Proposition~\ref{prop:all-deformations-are-standard} shows that there exists a \emph{unique} $p\in R$ such that $(E,b,j)$ is (uniquely) isomorphic to $(E(p,p),s_p,r)$, moreover this implies $A = R[[x,y]]/(xy-\pi)$ with $\pi = p^2$. Define $\varphi$ by $\tau \mapsto p$. Since the maps $s_p$ and $r$ are natural, the pushforward of $(E(\tau,\tau),s_\tau,r)$ is (uniquely) isomorphic to $(E(p,p),s_p,r)$. 
  
  Choosing any other map $\tau \mapsto q$ would give a root that is not isomorphic to $(E,b,j)$. Thus we have proven the existence and uniqueness of the map $\varphi$ of the desired form.
\end{proof}

\subsection{Deformations of nodes with multiple roots} \label{app:multiple_roots}

Taking a multiple root in the sense of Definition~\ref{def:multiple-root} and restricting it to the formal neighbourhood of a node in a family of curves warrants the study of the objects defined here. Fix a positive integer $m$, let $(A/R,\iota)$ be a deformation of the node as in Definition~\ref{def:std_node} and let $(E_i,b_i)_{i=1}^m$ be a sequence of roots on $A/R$. 

\begin{notation}
  We will reorder the roots so that we may assume there is an $m' \in \{0,\dots,m\}$ such that a root $E_i$ is free if and only if $i > m'$.
\end{notation}

\begin{definition}
  A sequence of maps $(\varphi_i \colon F \to E_i^2)_{i=1}^m$ which satisfy the following conditions are called a \emph{local pre-sync data}:
(1) if $E_i$ is not free then $\varphi_i$ is an isomorphism, 
(2) if $E_i$ are free then all $\varphi_i$ are isomorphisms, 
(3) for all $i,j$ we have $b_i \circ \varphi_i = b_j \circ \varphi_j$.
\end{definition}

\begin{remark}
  It will be often more convenient to identify the sequence $(\varphi_i)_{i=1}^m$ with a sequence of isomorphisms $(E_1^2 \to E_i^2)_{i=1}^{m'}$. The free roots do not need additional synchronization as the maps $b_i$ already identify their squares with the target.
\end{remark}

\begin{definition}
  Let $(\varphi_i)_{i=1}^m$ be a local pre-sync data. If for every $i,j \le m'$ the natural maps $\sym^2\sym^2 E_i \to \sym^4 E_j$ factor through an isomorphism $\sym^4 E_i \to \sym^4 E_j$ then $(\varphi_i)_{i=1}^m$ is called a \emph{local sync data}.
\end{definition}

Believing this can be no source of confusion, we will abuse notation and call a tuple $(E_i,b_i,\varphi_i)_{i=1}^m$ a \emph{multiple-root on $A$}, where $(\varphi_i)_{i=1}^m$ is a local sync data for the roots $(E_i,b_i)_{i=1}^m$. An \emph{isomorphism between two multiple-roots} $(E_i,b_i,\varphi_i \colon F \to E_i^2)_{i=1}^m$ and $(E_i',b_i',\varphi_i'\colon F' \to (E_i')^2)_{i=1}^m$ is a sequence of isomorphisms $(h_i \colon E_i \to E_i')_{i=1}^m$ for which there exists an isomorphism $f \colon F \to F' $ satisfying $h_i^2 \circ \varphi_i = \varphi_i' \circ f$.

Using the map $\iota \colon A \to \bar A$ we can define the push-forward $\iota_*(E_i,b_i,\varphi_i)_{i=1}^{m}$ to be the tuple $(\iota_*E_i,\iota_*b_i,\iota_* \varphi_i)_{i=1}^m$ with $\iota_* E_i = E_i\otimes_A \bar A$ and $\iota_* b_i$ as in Definition~\ref{not:iota_push} and $\iota_* \varphi_i \colon F\otimes_A \bar A \to \iota_* E_i$ the natural restriction of $\varphi_i$. 

Fix a multiple-root $\bar\xi\colonequals (\bar E_i, \bar b_i, \bar \varphi_i)_{i=1}^m$ on the standard node $\bar A$. Suppose $\xi\colonequals (E_i,b_i,\varphi_i)_{i=1}^m$ is a multiple-root on $A$ such that there exists a sequence of isomorphism $h_i \colon \iota_* E_i \to \bar E_i$ giving rise to an isomorphism between the multiple-roots $\iota_* \xi$ and $\bar \xi$. Then, the sequence of maps $j_i \colon E_i \to \iota_*E_i \overset{h_i}{\to} \bar E_i$ will be called a \emph{restriction map} of $\xi$ to $\bar\xi$. The tuple $(A/R,\iota,(E_i,b_i,\varphi_i,j_i)_{i=1}^m)$ will be called a \emph{deformation of $(\bar A/k,\bar\xi)$}. An isomorphism between two such deformations must commute with the restriction maps.

\begin{definition}\label{def:H}
  Let $H \colon \art_\Lambda \to \Sets$ be the functor assigning to each $R$ the set of isomorphism classes of deformations of $(\bar A/k ,\bar\xi)$.
\end{definition}

\begin{remark}\label{rem:replace_with_squares}
  Given a sync data $(\varphi_i)_{i=1}^m$, we can replace it with a sequence of isomorphisms $(E_1^2 \to E_i^2)_{i\le m'}$. 
\end{remark}

In order to construct a universal family for the functor $H$, we will reformulate the definition of local sync data by analyzing isomorphisms between squares of roots. Let $(E(p,p),s_p)$ and $(E(q,q),s_q)$ be standard roots on $A/R$ as in Definition~\ref{def:std_root}. Suppose $\varphi \colon E(p,p)^2 \isoto E(q,q)^2$ is an isomorphism satisfying $s_p = s_q \circ \varphi$.

\begin{proposition}\label{prop:iso_square}
  The map $\tphi \colon \sym^2\sym^2 E(p,p) \to \sym^4 E(q,q)$ induced by $\varphi$ factors through an isomorphism $\sym^4 E(p,p) \to \sym^4 E(q,q)$ if and only if there is an isomorphism $\mu\colon E(p,p) \to E(q,q)$ such that $\varphi = \mu^2$.
\end{proposition}
\begin{proof}
  Let $A^{\oplus 2} \to E(p,p)$ be the first step in the standard resolution (Definition~\ref{def:std_res}). Denote the images of $(1,0)$ and $(0,1)$ by $\xi_1,\xi_2 \in E(p,p)$ respectively. Similarly define $\zeta_1, \zeta_2 \in E(q,q)$. We need to characterize the condition that 
  \begin{equation}\label{eq:sym}
    \tphi(\xi_1^2\xi_2^2 - (\xi_1\xi_2)^2)=0. 
  \end{equation}

  Identify $\sym^2 A^{\oplus 2}$ with $A^{\oplus 3}$. Using standard arguments (\S A2.3~\cite{eisenbud-comm}) we find presentations of $E(p,p)^2$ and $E(q,q)^2$, for instance:
  \begin{equation}
    E(p,p)^{2} = \coker \left( L(p,p): A^{\oplus 4} \to A^{\oplus 3}  \right)
  \end{equation}
where 
\begin{equation}\label{eq:presentation}
  L(p,p) = %
  \begin{pmatrix}
    y  & -p & 0  & 0 \\
    -p & x  & y  & -p \\
    0  & 0  & -p & x
  \end{pmatrix}.
\end{equation}
We want to choose a simple lift of $\varphi$ to a map $A^{\oplus 3} \to A^{\oplus 3}$. Using the relations provided by $L(p,p)$ and $L(q,q)$ we may construct a lift $A^{\oplus 3} \to A^{\oplus 3}$ such that the corresponding $3\times 3$ matrix contains no terms involving $y$ in the first row, $x$ or $y$ in the second row, $x$ in the third row. In fact, such a lift is unique and we will denote it by $\hat\varphi$. It is easily seen that $\varphi$ commutes with $s_p$ and $s_q$ iff the lift can be expressed as 
\begin{equation}\label{eq:representation}
  \hat\varphi = %
  \begin{pmatrix}
    1  & 0 & 0  \\
    a_1 & u  & a_2   \\
    0  & 0  & 1 
  \end{pmatrix},
\end{equation}
where $u\in R^{\times}$ is such that $p=uq$ and $a_1,a_2 \in \ann(q) = \ann(p)$.

In terms of the entries of this matrix we may now write
\[
  \tphi(\xi_1^2\xi_2^2 - (\xi_1\xi_2)^2)= \zeta_1^2\zeta_2^2 + a_1\zeta_2^2 (\zeta_1\zeta_2) + a_2 \zeta_1^2(\zeta_1\zeta_2) + (a_1a_2-u^2)(\zeta_1\zeta_2)^2. 
\]
As before, we can calculate a presentation of $E(q,q)^4$. This presentation looks similar to Equation~\ref{eq:presentation} but with an extra block. It follows that the equality (\ref{eq:sym}) is satisfied iff $a_1=a_2=0$ and $u^2=1$. Since $R$ is a complete local ring of characteristic not two, we have $u^2=1$ iff  $u = \pm 1$.

Supposing (\ref{eq:sym}) holds, we conclude $p=uq=\pm q$. By Lemma~\ref{lem:iso-class} we can find isomorphisms $\mu$ between the two roots. Moreover, $\mu$ admits a representation $\begin{bsmallmatrix} \varepsilon_1 & 0  \\  0 & \varepsilon_2  \end{bsmallmatrix}$ where $\varepsilon_1,\varepsilon_2 \in \{-1,1\}$ and $\varepsilon_1\varepsilon_2 = u$. It is now clear that $\varphi=\mu^2$. Conversely, if $\varphi=\mu^2$ then the representation $\hat\varphi$ is of the required form and (\ref{eq:sym}) holds.
\end{proof}

\begin{definition}\label{def:parity}
  Let $\varphi$ be the square of an isomorphism $\mu\colon (E(p,p),s_p) \to (E(q,q),s_q)$. Then the entry $u \in \{\pm 1\}$ appearing in (\ref{eq:representation}) will be called the \emph{parity of $\varphi$}. If $p=q$ and $\mu = \pm \id$ then $u=1$ otherwise $u=-1$.
\end{definition}

\begin{remark}\label{rem:central_fiber}
  Observe that the parity of $\varphi$ can be determined by restricting $\varphi$ to the central fiber $\bar A/k$. In particular, the root with multiple roots $(\bar A /k , \bar \xi)$ completely determines the local sync data for any of its deformations.
\end{remark}

\begin{notation}\label{not:root_convention}
  For our fixed root $\bar \xi$ on $\bar A/k$ let us identify the non-free roots with the standard root $(\bar E,\bar s) = (E(0,0),s_0)$. We can and will choose this identification so that the parity of each of the isomorphisms $(\bar E_1^2 \to \bar E_i^2)_{i=1}^{m'}$ induced by the local sync data $(\bar \varphi_i)_{i=1}^{m'}$ is 1.
\end{notation}

The following lemma shows how we can put a deformations of $\bar \xi$ into standard form. In light of the Remark~\ref{rem:central_fiber} above, we will omit the local sync data of these deformations as they are already determined by the local sync data of $\bar \xi$. 

\begin{lemma} \label{lem:unique-p-for-multiple-roots}
  Let $(A/R,\iota,(E_i,b_i,\varphi_i,j_i)_{i=1}^m)$ be a deformation of $(\bar A/k,\bar\xi)$. Then $\exists ! p \in R$ such that for all $i\le m'$ we have $(E_i,b_i,j_i) = (E(p, p),s_p,r)$ with $r$ the natural restriction map from Definition~\ref{def:restriction_map}.
\end{lemma}

\begin{proof}
  By Proposition~\ref{prop:all-deformations-are-standard} we know that $\forall i \le m'$ $\exists ! p_i \in R$ such that $(E_i,b_i,j_i)=(E(p_i,p_i),s,r)$. Let $p = p_1$. From Proposition~\ref{prop:iso_square} we observe that all roots $(E_i,b_i)$ must be isomorphic and by Lemma~\ref{lem:iso-class} we conclude $p_i \in \set{\pm p}$. 
  
  If $p=0$ then there is nothing more to prove so assume $p \neq 0$. Then the sign of $p_i$ is completely determined by the local sync data and in particular by the sign of $u$ in (\ref{eq:representation}). We choose the identification of the roots on the central fiber so that this sign is always $1$.
\end{proof}

Since free roots have no non-trivial deformations and they face no obstructions in deforming, if all $m$ roots in $\bar \xi$ are free ($m'=0$) then the deformation functor $H$ of $(\bar A/k,\bar \xi)$ agrees with the deformation functor of $\bar A /k$ alone (Lemma~\ref{lem:free_roots_deform_trivially}). Therefore, we will now assume $m' > 0$.

\begin{theorem}\label{thm:universal-def-of-mult-roots}
  The functor of infinitesimal deformations $H$ of $(\bar A/k , \bar \xi)$ is pro-represented by $\Lambda[[\tau]]$ with the universal deformation given by $(\Lambda[[\tau]] \to \Lambda[[\tau,x,y]]/(xy-\tau^2),\tau\mapsto 0)$ with the first $m'$ roots equal to $(E(\tau,\tau),s_\tau,r)$ and the remaining roots free. The local sync data are trivially obtained from those of $\bar \xi$, which by our Notation~\ref{not:root_convention} correspond to the squares of the identity maps. 
\end{theorem}

\begin{proof}
  Let $(A/R,\iota)$ be a deformation of the node and let $(\cR,\xj) = (E_i,b_i,j_i)_{i=1}^m$ together with the local sync data $(\varphi_i)_{i=1}^m$ be a deformation of $(\bar A/k , \bar \xi)$. Any map $\nu\colon \Lambda[[\tau]] \to R$ is uniquely defined by the choice of $p \in R$ for which $\tau \mapsto p$.  Lemma~\ref{lem:unique-p-for-multiple-roots} tells us that there is a unique $p \in R$ for which $\nu_*(E( \tau , \tau),s_\tau,r) = (E_i,b_i,j_i)$. This proves the existence and uniqueness of $\nu$ provided we show that the synchronizations agree. This is done by reducing to the central fiber, where the compatibility of the synchronizations is immediate.
\end{proof}

\subsection{Formal neighbourhoods in general}

Let $k$ be any field. Then by Cohen structure theorem there exists a \emph{universal coefficient ring}, which we will denote by $\xo_k$, so that any complete local ring with residue field $k$ contains a copy of $\xo_k$. If $\chr k = 0$ then $\xo_k = k$.

Let $S$ be scheme over $k$ and let $s \colon\spec k \to S$ be $k$-valued point. The usual complete local ring $\hat\co_{S,s}$ pro-represents the functor $\art_{\xo_k} \to \Sets$ defined by $A \mapsto \hom_s(A,S)$, where the subscript $s$ indicates that the morphisms must restrict to $s$ on the residue field. This allows one to generalize the definition of a complete local ring to more general situations, for instance to geometric points of Deligne--Mumford stacks.

If $k'$ is any field and $s' \colon \spec k' \to S$ is a $k'$-valued point of $S$ then we can still define the functor $Q_{s'} \colon \art_{\xo_{k'}} \to \Sets$ via the rule $A \mapsto \hom_{s'}(A,S)$. If $s'$ factors through a $k$-valued point $s \colon \spec k \to S$ then $Q_{s'}$ is pro-represented by the complete local ring $\hat \co_{S,s} \otimes_{\xo_{k}} \xo_{k'}$.  

For a Deligne--Mumford stack $\cm$ and a point $p \colon \spec k \to \cm$ the functor $Q_p \colon \art_{\xo_{k}} \to \Sets$ can be defined as above. This functor is seen to be pro-representable by using any \'etale chart. 

\begin{definition}\label{def:local_rings}
  Let $\cm$ be a Deligne--Mumford stack and let $p \colon \spec k \to \cm$ be any $k$-valued point. The complete local ring pro-representing $Q_p$ will be denoted by $\hat \co_{\cm,p}$. The functor $Q_p$ is called the \emph{local deformation functor of $p$} and $\spec \co_{\cm,p}$ is called the \emph{formal neighbourhood} of $p$. 
\end{definition}

\printbibliography
\vfill
\end{document}